\documentclass{article}
\usepackage{amsmath}
\usepackage{amsfonts}
\usepackage{graphicx}
\usepackage{hyperref}
\usepackage{cite}

\newtheorem{theorem}{Theorem}

\newtheorem{definition}[theorem]{Definition}

\newtheorem{lemma}[theorem]{Lemma}
\newtheorem{notation}[theorem]{Notation}

\newtheorem{remark}[theorem]{Remark}

\newenvironment{proof}[1][Proof]{\noindent\textbf{#1.} }{\ \rule{0.5em}{0.5em}}

\begin{document}

\title{Part 2. Infinite series and logarithmic integrals associated to
differentiation with respect to parameters of the Whittaker $\mathrm{W}%
_{\kappa ,\mu }\left( x\right) $ function.}
\author{Alexander Apelblat$^{1}$, Juan Luis Gonz\'{a}lez-Santander$^{2}$. \\
$^{1}$ Department of Chemical Engineering, \\
Ben Gurion University of the
Negev, \\
84105 Beer Sheva, 84105, Israel. apelblat@bgu.ac.il\\
$^{2}$ Department of Mathematics, Universidad de Oviedo, \\
33007 Oviedo, Spain. gonzalezmarjuan@uniovi.es}
\maketitle

\begin{abstract}
First derivatives with respect to the parameters of the Whittaker function $%
\mathrm{W}_{\kappa ,\mu }\left( x\right) $ are calculated. Using the
confluent hypergeometric function, these derivarives can be expressed as
infinite sums of quotients of the digamma and gamma functions. Also, it is
possible to obtain these parameter derivatives in terms of infinite
integrals with integrands containing elementary functions (products of
algebraic, exponential and logarithmic functions) from the integral
representation of $\mathrm{W}_{\kappa ,\mu }\left( x\right) $. These
infinite sums and integrals can be expressed in closed-form for particular
values of the parameters. Finally, an integral representation of the
integral Whittaker function $\mathrm{wi}_{\kappa ,\mu }\left( x\right) $ and
its derivative with respect to $\kappa $, as well as some reduction formulas
for the integral Whittaker functions $\mathrm{Wi}_{\kappa ,\mu }\left(
x\right) $ and $\mathrm{wi}_{\kappa ,\mu }\left( x\right) $ are calculated.
\end{abstract}


\textbf{AMS\ Subject Classification}:\ 33B15, 33B20, 33C10, 33C15, 33C20,
33C50, 33E20.

\section{Introduction}

Two functions $\mathrm{M}_{\kappa ,\mu }\left( x\right) $ and $\mathrm{W}%
_{\kappa ,\mu }\left( x\right) $ were introduced to the mathematical
literature by Whittaker \cite{whittaker1903expression} in 1903, and they are
linearly independent solutions of the following second order differential
equation:%
\begin{eqnarray*}
&&\frac{d^{2}y}{dx^{2}}+\left( \frac{\frac{1}{4}-\mu }{x^{2}}+\frac{\kappa }{%
x}-\frac{1}{4}\right) y=0, \\
&&y\left( x\right) =C_{1}\,\mathrm{M}_{\kappa ,\mu }\left( x\right) +C_{2}\,%
\mathrm{W}_{\kappa ,\mu }\left( x\right) , \\
&&2\mu \neq -1,-2,\ldots
\end{eqnarray*}%
where $\kappa $ and $\mu $ are parameters. For particular values of these
parameters, the Whittaker functions $\mathrm{M}_{\kappa ,\mu }\left(
x\right) $ and $\mathrm{W}_{\kappa ,\mu }\left( x\right) $ can be reduced to
a variety of elementary and special functions (such as modified Bessel
functions, incomplete gamma functions, parabolic cylinder functions, error
functions, logarithmic and cosine integrals, as well as the generalized
Hermite and Laguerre polynomials). Recently, Mainardi et al. \cite%
{mainardi2022wright}\ investigated the special case wherein the Wright
function can be expressed in terms of Whittaker functions.

The Whittaker functions can be expressed as \cite[Eqn. 13.14.2]%
{olver2010nist}:\ \
\begin{eqnarray}
\mathrm{M}_{\kappa ,\mu }\left( z\right) &=&z^{\mu
+1/2}e^{-z/2}\,_{1}F_{1}\left( \left.
\begin{array}{c}
\frac{1}{2}+\mu -\kappa \\
1+2\mu%
\end{array}%
\right\vert z\right)  \label{M_k,mu_def} \\
2\mu &\neq &-1,-2,\ldots  \notag
\end{eqnarray}%
and \cite[Eqn. 13.14.33]{olver2010nist}:%
\begin{eqnarray}
\mathrm{W}_{\kappa ,\mu }\left( z\right) &=&\frac{\Gamma \left( -2\mu
\right) }{\Gamma \left( \frac{1}{2}-\mu -\kappa \right) }\mathrm{M}_{\kappa
,\mu }\left( z\right) +\frac{\Gamma \left( 2\mu \right) }{\Gamma \left(
\frac{1}{2}+\mu -\kappa \right) }\mathrm{M}_{\kappa ,-\mu }\left( z\right) ,
\label{W_k,mu_def} \\
2\mu &\notin &%
\mathbb{Z}
,  \notag
\end{eqnarray}%
where $\Gamma \left( x\right) $ denotes the \textit{gamma function,} and the
\textit{Kummer function} is defined as \cite[Eqn. 47:3:1]{oldham2009atlas}:%
\begin{equation}
_{1}F_{1}\left( \left.
\begin{array}{c}
a \\
b%
\end{array}%
\right\vert z\right) =\sum_{n=0}^{\infty }\frac{\left( a\right) _{n}}{\left(
b\right) _{n}}\frac{z^{n}}{n!},  \label{1F1_Whittaker_def}
\end{equation}%
where $\left( \alpha \right) _{n}=$ $\frac{\Gamma \left( \alpha +n\right) }{%
\Gamma \left( \alpha \right) }$ denotes the \textit{Pochhammer polynomial}
and
\begin{equation}
_{p}F_{q}\left( \left.
\begin{array}{c}
a_{1},\ldots ,a_{p} \\
b_{1},\ldots ,b_{q}%
\end{array}%
\right\vert x\right) =\sum_{n=0}^{\infty }\frac{\left( a_{1}\right)
_{n}\cdots \left( a_{p}\right) _{n}}{\left( b_{1}\right) _{n}\cdots \left(
b_{q}\right) _{n}}\frac{x^{n}}{n!},  \label{pFq_def}
\end{equation}%
is the \textit{generalized hypergeometric function}.

Also, the Whittaker function $\mathrm{W}_{\kappa ,\mu }\left( x\right) $ can
be expressed as \cite[Eqn. 13.14.3]{olver2010nist}:\
\begin{equation}
\mathrm{W}_{\kappa ,\mu }\left( z\right) =e^{-z/2}z^{\mu +1/2}\,\mathrm{U}%
\left( \frac{1}{2}+\mu -\kappa ,1+2\mu ,z\right) ,  \label{W_def_Tricomi}
\end{equation}%
where $\mathrm{U}\left( a,b,z\right) $ denotes the \textit{Tricomi function}.

Analytical properties of the Whittaker functions (see \cite%
{erdelyi1953bateman,slater1960confluent,whittaker1920course,olver2010nist,magnus2013formulas,buchholz1969confluent,gradstein2015table,prudnikov1986integrals,oldham2009atlas}%
)\ are of great interest in Mathematical Physics because these functions are
involved in many applications, such as the solutions of the wave equation in
paraboloidal coordinates, the behaviour of charges particles in fields with
Coulomb potentials, stationary Green's function in atomic and molecular
calculations in Quantum Mechanics (i.e. solution of Schr\"{o}dinger equation
for the harmonic oscillator), probability density functions, and in many
other physical and engineering problems \cite%
{slater1960confluent,laurenzi1973derivatives,lagarias2009schrodinger,omair2022family}%
. \

Mostly, the Whittaker functions are regarded as a function of variable $x$
with fixed values of parameters $\kappa $ and $\mu $, although there are few
investigations where mathematical operations associated with both parameters
are considered, especially for the $\kappa $ parameter \cite%
{laurenzi1973derivatives,buschman1974finite,abad2003successive,becker2009infinite}%
. In this context, it is worthwhile to mention Laurenzi's paper \cite%
{laurenzi1973derivatives}, where the calculation of the derivative of $%
\mathrm{W}_{\kappa ,1/2}\left( x\right) $ with respect to $\kappa $ when
this parameter is an integrer is derived. In \cite{buschman1974finite},
Buschman showed that the derivative of $\mathrm{W}_{\kappa ,\mu }\left(
x\right) $ with respect to the parameters can be expressed in terms of
finite sums of these $\mathrm{W}_{\kappa ,\mu }\left( x\right) $ functions.
Higher derivatives of the Whittaker functions with respect to parameter $%
\kappa $ were discussed by Abad and Sesma \cite{abad2003successive}, and
integrals with respect to parameter $\mu $ by Becker \cite%
{becker2009infinite}. Since the Whittaker functions are related to the
confluent hypergeometric function, it is worth mention the investigation of
the derivatives of the generalized hypergeometric functions presented by
Ancarini and Gasaneo \cite{ancarani2010derivatives} or Sofostasios and
Brychkov \cite{sofotasios2018derivatives}.

The integral Whittaker functions were introduced by us \cite%
{apelblat2021integral} as follows:
\begin{eqnarray}
\mathrm{Wi}_{\kappa ,\mu }\left( x\right) &=&\int_{0}^{x}\frac{\mathrm{W}%
_{\kappa ,\mu }\left( t\right) }{t}dt,  \label{Wi_def} \\
\mathrm{wi}_{\kappa ,\mu }\left( x\right) &=&\int_{x}^{\infty }\frac{\mathrm{%
W}_{\kappa ,\mu }\left( t\right) }{t}dt.  \label{wi__def}
\end{eqnarray}

In the current paper, the main attention will be devoted to Whittaker
function $\mathrm{W}_{\kappa ,\mu }\left( x\right) $ by analyzing the first
derivative of this function with respect to the parameters from the
corresponding series and integral representations. Direct differentiation of
the Whittaker functions leads to infinite sums of quotients of the digamma
and gamma functions. It is possible to calculate these sums in closed-form
in some cases with the aid of MATHEMATICA\ program. When the integral
representations of the Whittaker function $\mathrm{W}_{\kappa ,\mu }\left(
x\right) $ are taken into account, the results of differentiation can be
expressed in terms of Laplace transforms of elementary functions. Integrands
of the these Laplace type integrals include products of algebraic,
exponential and logarithmic functions. New groups of infinite integrals are
comparable to those investigated by K\"{o}lbig \cite{kolbig1987integral},
Geddes et al. \cite{geddes1990evaluation}, and Apelblat and Kravitzky \cite%
{apelblat1985integral} are calculated in this paper.

Also, we will focus our attention on the integral Whittaker functions $%
\mathrm{Wi}_{\kappa ,\mu }\left( x\right) $ and $\mathrm{wi}_{\kappa ,\mu
}\left( x\right) $ in order to derive some new reduction formulas, as well
as an integral representation of $\mathrm{wi}_{\kappa ,\mu }\left( x\right) $
and its first derivative with respect to parameter $\kappa $.

\section{Parameter differentiation of $\mathrm{W}_{\protect\kappa ,\protect%
\mu }$ via Kummer function $_{1}F_{1}$}

\begin{notation}
Unless indicated otherwise, it is assumed throughout the paper that $x$ is a
real variable and $z$ is a complex variable.
\end{notation}

\begin{definition}
According to the notation introduced by Ancarini and Gasaneo \cite%
{ancarani2010derivatives}, define%
\begin{equation}
G^{\left( 1\right) }\left( \left.
\begin{array}{c}
a \\
b%
\end{array}%
\right\vert x\right) =\frac{\partial }{\partial a}\,\left[ _{1}F_{1}\left(
\left.
\begin{array}{c}
a \\
b%
\end{array}%
\right\vert x\right) \right] ,  \label{G(1)_def}
\end{equation}%
and%
\begin{equation}
H^{\left( 1\right) }\left( \left.
\begin{array}{c}
a \\
b%
\end{array}%
\right\vert x\right) =\frac{\partial }{\partial b}\,\left[ _{1}F_{1}\left(
\left.
\begin{array}{c}
a \\
b%
\end{array}%
\right\vert x\right) \right] .  \label{H(1)_def}
\end{equation}
\end{definition}

\subsection{Derivative with respect to the first parameter $\partial \mathrm{%
W}_{\protect\kappa ,\protect\mu }\left( x\right) /\partial \protect\kappa $}

Taking into account (\ref{M_k,mu_def}) and (\ref{G(1)_def}), direct
differentiation of (\ref{W_k,mu_def}) yields:%
\begin{eqnarray}
&&\frac{\partial \mathrm{W}_{\kappa ,\mu }\left( x\right) }{\partial \kappa }
\label{DkW_(1)} \\
&=&\frac{\Gamma \left( -2\mu \right) }{\Gamma \left( \frac{1}{2}-\mu -\kappa
\right) }\left[ \psi \left( \frac{1}{2}-\mu -\kappa \right) \mathrm{M}%
_{\kappa ,\mu }\left( x\right) -x^{1/2+\mu }e^{-x/2}\,G^{\left( 1\right)
}\left( \left.
\begin{array}{c}
\frac{1}{2}+\mu -\kappa \\
1+2\mu%
\end{array}%
\right\vert x\right) \right]  \notag \\
&&+\frac{\Gamma \left( 2\mu \right) }{\Gamma \left( \frac{1}{2}+\mu -\kappa
\right) }\left[ \psi \left( \frac{1}{2}+\mu -\kappa \right) \mathrm{M}%
_{\kappa ,-\mu }\left( x\right) -x^{1/2-\mu }e^{-x/2}\,G^{\left( 1\right)
}\left( \left.
\begin{array}{c}
\frac{1}{2}-\mu -\kappa \\
1-2\mu%
\end{array}%
\right\vert x\right) \right] .  \notag
\end{eqnarray}

If we apply first Kummer's transformation formula \cite[Eqn. 13.2.39]%
{olver2010nist}:
\begin{equation}
_{1}F_{1}\left( \left.
\begin{array}{c}
a \\
b%
\end{array}%
\right\vert x\right) =e^{x}\,_{1}F_{1}\left( \left.
\begin{array}{c}
b-a \\
b%
\end{array}%
\right\vert -x\right) ,  \label{Kummer_transform}
\end{equation}%
we can rewrite (\ref{DkW_(1)})\ as%
\begin{eqnarray}
&&\frac{\partial \mathrm{W}_{\kappa ,\mu }\left( x\right) }{\partial \kappa }
\label{DkW_(2)} \\
&=&\frac{\Gamma \left( -2\mu \right) }{\Gamma \left( \frac{1}{2}-\mu -\kappa
\right) }\left[ \psi \left( \frac{1}{2}-\mu -\kappa \right) \mathrm{M}%
_{\kappa ,\mu }\left( x\right) +x^{1/2+\mu }e^{x/2}\,G^{\left( 1\right)
}\left( \left.
\begin{array}{c}
\frac{1}{2}+\mu +\kappa \\
1+2\mu%
\end{array}%
\right\vert -x\right) \right]  \notag \\
&&+\frac{\Gamma \left( 2\mu \right) }{\Gamma \left( \frac{1}{2}+\mu -\kappa
\right) }\left[ \psi \left( \frac{1}{2}+\mu -\kappa \right) \mathrm{M}%
_{\kappa ,-\mu }\left( x\right) -x^{1/2-\mu }e^{-x/2}\,G^{\left( 1\right)
}\left( \left.
\begin{array}{c}
\frac{1}{2}-\mu -\kappa \\
1-2\mu%
\end{array}%
\right\vert x\right) \right] .  \notag
\end{eqnarray}

\begin{theorem}
\label{Theorem_1}For $2\mu \notin
\mathbb{Z}
$, the following parameter derivative formula of $\mathrm{W}_{\kappa ,\mu
}\left( x\right) $ holds true:%
\begin{eqnarray}
&&\left. \frac{\partial \mathrm{W}_{\kappa ,\pm \mu }\left( x\right) }{%
\partial \kappa }\right\vert _{\kappa =\mu +1/2}=\sqrt{x}e^{-x/2}
\label{DkW_k=mu+1/2} \\
&&\left\{ x^{\mu }\left[ \psi \left( -2\mu \right) -\frac{x}{2\mu +1}%
\,_{2}F_{2}\left( \left.
\begin{array}{c}
1,1 \\
2\mu +2,2%
\end{array}%
\right\vert x\right) \right] +\Gamma \left( 2\mu +1\right) x^{-\mu }\left(
-x\right) ^{2\mu }\gamma \left( -2\mu ,-x\right) \right\} ,  \notag
\end{eqnarray}%
where $\gamma \left( \nu ,z\right) $ denotes the lower incomplete gamma
function (\ref{gamma_def}).
\end{theorem}

\begin{proof}
First note that
\begin{equation}
\frac{\partial \mathrm{W}_{\kappa ,\mu }\left( x\right) }{\partial \kappa }=%
\frac{\partial \mathrm{W}_{\kappa ,-\mu }\left( x\right) }{\partial \kappa },
\label{DkW_mu=DkW_-mu}
\end{equation}%
since \cite[Eqn. 13.14.31]{olver2010nist}:
\begin{equation}
\mathrm{W}_{\kappa ,\mu }\left( x\right) =\mathrm{W}_{\kappa ,-\mu }\left(
x\right) .  \label{W_mu=W_-mu}
\end{equation}%
Now, let us calculate $\left. \partial \mathrm{W}_{\kappa ,\mu }\left(
x\right) /\partial \kappa \right\vert _{\kappa =\mu +1/2}$. For this
purpose, take $\kappa =\mu +1/2-\epsilon $ in (\ref{DkW_(2)})\ to obtain%
\begin{eqnarray}
&&\left. \frac{\partial \mathrm{W}_{\kappa ,\mu }\left( x\right) }{\partial
\kappa }\right\vert _{\kappa =\mu +1/2-\epsilon }  \label{W_epsilon} \\
&=&\frac{\Gamma \left( -2\mu \right) }{\Gamma \left( -2\mu +\epsilon \right)
}\left[ \psi \left( -2\mu +\epsilon \right) \mathrm{M}_{\mu +1/2-\epsilon
,\mu }\left( x\right) +x^{1/2+\mu }e^{x/2}\,G^{\left( 1\right) }\left(
\left.
\begin{array}{c}
1+2\mu -\epsilon \\
1+2\mu%
\end{array}%
\right\vert -x\right) \right]  \notag \\
&&+\frac{\Gamma \left( 2\mu \right) }{\Gamma \left( \epsilon \right) }\left[
\psi \left( \epsilon \right) \mathrm{M}_{\mu +1/2-\epsilon ,-\mu }\left(
x\right) -x^{1/2-\mu }e^{-x/2}\,G^{\left( 1\right) }\left( \left.
\begin{array}{c}
-2\mu +\epsilon \\
1-2\mu%
\end{array}%
\right\vert x\right) \right] .  \notag
\end{eqnarray}%
Note that according to \cite[Eqn. 13.18.2]{olver2010nist}
\begin{equation}
\mathrm{M}_{\mu +1/2,\mu }\left( x\right) =e^{-x/2}x^{1/2+\mu }.
\label{M_mu+1/2,mu}
\end{equation}%
Also, from (\ref{M_k,mu_def})\ and (\ref{Kummer_transform}), we have%
\begin{eqnarray}
\mathrm{M}_{\mu +1/2,-\mu }\left( x\right) &=&e^{x/2}x^{1/2-\mu
}\,_{1}F_{1}\left( \left.
\begin{array}{c}
1 \\
1+2\mu%
\end{array}%
\right\vert -x\right)  \notag \\
&=&e^{x/2}x^{1/2-\mu }\sum_{n=0}^{\infty }\frac{\left( -x\right) ^{n}}{%
\left( 1-2\mu \right) _{n}}.  \label{M_-mu_(1)}
\end{eqnarray}%
Taking into account \cite[Eqn. 45:6:2]{oldham2009atlas}:
\begin{equation*}
e^{x}\gamma \left( \nu ,x\right) =\frac{x^{\nu }}{\nu }\sum_{n=0}^{\infty }%
\frac{x^{n}}{\left( 1+\nu \right) _{n}},
\end{equation*}%
rewrite (\ref{M_-mu_(1)})\ as%
\begin{equation}
\mathrm{M}_{\mu +1/2,-\mu }\left( x\right) =-2\mu \,e^{-x/2}x^{1/2-\mu
}\left( -x\right) ^{2\mu }\gamma \left( -2\mu ,-x\right) .
\label{M_mu+1/2,-mu}
\end{equation}%
Consider as well the reduction formula given in the Appendix \ref%
{G(1)_(a;a;x)}:%
\begin{equation}
G^{\left( 1\right) }\left( \left.
\begin{array}{c}
a \\
a%
\end{array}%
\right\vert x\right) =\frac{x\,e^{x}}{a}\,_{2}F_{2}\left( \left.
\begin{array}{c}
1,1 \\
a+1,2%
\end{array}%
\right\vert -x\right) .  \label{G1(a;a;x)}
\end{equation}%
Finally, according to the property \cite[Eqn. 44:5:3]{oldham2009atlas}:
\begin{equation*}
\psi \left( z+1\right) =\frac{1}{z}+\psi \left( z\right) ,
\end{equation*}%
see that%
\begin{equation}
\lim_{\epsilon \rightarrow 0}\frac{\psi \left( \epsilon \right) }{\Gamma
\left( \epsilon \right) }=\lim_{\epsilon \rightarrow 0}\frac{1}{\Gamma
\left( \epsilon \right) }\left[ \psi \left( \epsilon +1\right) -\frac{1}{%
\epsilon }\right] =-1.  \label{Limit_Psi/Gamma}
\end{equation}%
Now, take the limit $\epsilon \rightarrow 0$ in (\ref{W_epsilon}),
considering the results given in (\ref{DkW_mu=DkW_-mu}), (\ref{M_mu+1/2,mu}%
), (\ref{M_mu+1/2,-mu}), (\ref{G1(a;a;x)}) and (\ref{Limit_Psi/Gamma}), to
obtain (\ref{DkW_k=mu+1/2}), as we wanted to prove.
\end{proof}

Table \ref{Table_1}\ presents some explicit expressions for particular
values of (\ref{DkW_k=mu+1/2}), obtained with the help of MATHEMATICA\
program.

\begin{center}
\begin{table}[tbp] \centering%
\caption{Derivative of $\mathrm{W}_{\kappa,\mu}$ with respect
to $\kappa$ by using (\ref{DkW_k=mu+1/2}).}%
\rotatebox{90}{
\begin{tabular}{|c|c|c|}
\hline
$\kappa $ & $\mu $ & $\frac{\partial \mathrm{W}_{\kappa ,\mu }\left(
x\right) }{\partial \kappa }$ \\ \hline\hline
$-\frac{3}{4}$ & $\pm \frac{5}{4}$ & $\frac{1}{3}\,x^{-3/4}e^{-x/2}\left[
2x\,_{2}F_{2}\left( 1,1;-\frac{1}{2},2;x\right) +3\pi \,\mathrm{erfi}\left(
\sqrt{x}\right) +2\sqrt{\pi \,x}\,e^{x}\left( 2x-3\right) -3\gamma +8-3\ln 4%
\right] $ \\ \hline
$-\frac{1}{4}$ & $\pm \frac{3}{4}$ & $x^{-1/4}e^{-x/2}\left[
2x\,_{2}F_{2}\left( 1,1;\frac{1}{2},2;x\right) +\pi \,\mathrm{erfi}\left(
\sqrt{x}\right) -2\sqrt{\pi \,x}\,e^{x}-\gamma +2-\ln 4\right] $ \\ \hline
$-\frac{1}{6}$ & $\pm \frac{2}{3}$ & $%
\begin{array}{l}
\frac{1}{6}x^{-5/6}e^{-x/2}\left\{ 3x^{2/3}\left[ 6x\,_{2}F_{2}\left( 1,1;%
\frac{2}{3},2;x\right) -2\gamma +6-3\ln 3\right] \right. \\
\quad \left. -6x^{2}\Gamma \left( -\frac{1}{3}\right) \mathrm{E}%
_{-1/3}\left( -x\right) -\sqrt{3}\pi \left[ x^{2/3}+4\left( -x\right) ^{2/3}%
\right] \right\}%
\end{array}%
$ \\ \hline
$\frac{1}{6}$ & $\pm \frac{1}{3}$ & $%
\begin{array}{l}
\frac{1}{6}x^{-1/6}e^{-x/2}\left\{ -3x^{1/3}\left[ 6x\,_{2}F_{2}\left( 1,1;%
\frac{4}{3},2;x\right) +2\gamma +3\ln 3\right] \right. \\
\quad \left. -6x\,\Gamma \left( \frac{1}{3}\right) \mathrm{E}_{1/3}\left(
-x\right) +\sqrt{3}\pi \left[ x^{1/3}-4\left( -x\right) ^{1/3}\right]
\right\}%
\end{array}%
$ \\ \hline
$\frac{1}{4}$ & $\pm \frac{1}{4}$ & $-x^{1/4}e^{-x/2}\left[
2x\,_{2}F_{2}\left( 1,1;\frac{3}{2},2;x\right) -\pi \,\mathrm{erfi}\left(
\sqrt{x}\right) +\gamma +\ln 4\right] $ \\ \hline
$\frac{3}{4}$ & $\pm \frac{1}{4}$ & $\frac{1}{3}e^{-x/2}\left\{ x^{3/4}\left[
-2x\,_{2}F_{2}\left( 1,1;\frac{5}{2},2;x\right) +3\left( \pi \,\mathrm{erfi}%
\left( \sqrt{x}\right) -\gamma +2-\ln 4\right) \right] -3\sqrt{\pi }%
\,x^{1/4}e^{x}\right\} $ \\ \hline
$\frac{5}{6}$ & $\pm \frac{1}{3}$ & $%
\begin{array}{l}
\frac{1}{30}x^{1/6}e^{-x/2}\left\{ -18\,x^{5/3}\,_{2}F_{2}\left( 1,1;\frac{8%
}{3},2;x\right) +15\,x^{2/3}\left( 3-2\gamma -3\ln 3\right) \right. \\
\quad -\left. 30\,\Gamma \left( \frac{5}{3}\right) \mathrm{E}_{5/3}\left(
-x\right) -5\sqrt{3}\pi \left[ x^{2/3}+4\left( -x\right) ^{1/3}\right]
\right\}%
\end{array}%
$ \\ \hline
$\frac{5}{4}$ & $\pm \frac{3}{4}$ & $\frac{1}{30}x^{-1/4}e^{-x/2}\left\{
-2x^{3/2}\left[ 6x\,_{2}F_{2}\left( 1,1;\frac{7}{2},2;x\right) -5\left( \pi
\,\mathrm{erfi}\left( \sqrt{x}\right) -3\gamma +8-3\ln 4\right) \right] -15%
\sqrt{\pi }\,e^{x}\left( 2x+1\right) \right\} $ \\ \hline
\end{tabular}%
}
\label{Table_1}%
\end{table}%
\end{center}

Next, we present other reduction formula of $\partial \mathrm{W}_{\kappa
,\mu }\left( x\right) /\partial \kappa $ from the result found in \cite%
{laurenzi1973derivatives}.

\begin{theorem}
The following reduction formula holds true for $n=1,2,\ldots $%
\begin{eqnarray}
&&\left. \frac{\partial \mathrm{W}_{\kappa ,\pm 1/2}\left( x\right) }{%
\partial \kappa }\right\vert _{\kappa =n}  \label{DkW_n,1/2} \\
&=&\left( -1\right) ^{n}\left( n-1\right) !e^{-x/2}\left[ \sum_{\ell
=0}^{n-1}\frac{n-\ell }{n+\ell }\,L_{\ell }^{\left( -1\right) }\left(
x\right) +n\,L_{\ell }^{\left( -1\right) }\left( x\right) \ln x\,\right] ,
\notag
\end{eqnarray}%
where $L_{n}^{\left( \alpha \right) }\left( x\right) $ denotes the Laguerre
polynomial.
\end{theorem}

\begin{proof}
First note that, according to (\ref{DkW_mu=DkW_-mu}), we have
\begin{equation}
\frac{\partial \mathrm{W}_{\kappa ,1/2}\left( x\right) }{\partial \kappa }=%
\frac{\partial \mathrm{W}_{\kappa ,-1/2}\left( x\right) }{\partial \kappa }.
\label{DkW_1/2=DkW_-1/2}
\end{equation}%
Therefore, let us calculte $\partial \mathrm{W}_{\kappa ,1/2}\left( x\right)
/\partial \kappa $. For this purpose, consider the formula \cite%
{laurenzi1973derivatives}:%
\begin{eqnarray}
&&\left. \frac{\partial \mathrm{W}_{\kappa ,1/2}\left( x\right) }{\partial
\kappa }\right\vert _{\kappa =n}  \label{DkW_Laurenzi} \\
&=&\left( -1\right) ^{n}\left( n-1\right) !\sum_{\ell =0}^{n-1}\frac{\left(
-1\right) ^{\ell }\left( n-\ell \right) }{\ell !\left( n+\ell \right) }%
\mathrm{W}_{\ell ,1/2}\left( x\right) +\mathrm{W}_{n,1/2}\left( x\right) \ln
x\,  \notag
\end{eqnarray}%
Also, from \cite[Eqn. 13.18.17]{olver2010nist}, we have for $n=0,1,2,\ldots $
\begin{equation}
\mathrm{W}_{\kappa +n,\kappa -1/2}\left( x\right) =\left( -1\right)
^{n}n!e^{-x/2}x^{\kappa }L_{n}^{\left( 2\kappa -1\right) }\left( x\right) ,
\label{W_k+n_reduction}
\end{equation}%
thus applying (\ref{W_mu=W_-mu}) and taking $\kappa =0$ in (\ref%
{W_k+n_reduction}), we have%
\begin{equation}
\mathrm{W}_{n,1/2}\left( x\right) =\mathrm{W}_{n,-1/2}\left( x\right)
=\left( -1\right) ^{n}n!\,e^{-x/2}L_{n}^{\left( -1\right) }\left( x\right) .
\label{W_n_1/2}
\end{equation}%
Finally, insert (\ref{W_n_1/2})\ into (\ref{DkW_n,1/2})\ and consider (\ref%
{DkW_1/2=DkW_-1/2})\ to obtain (\ref{DkW_n,1/2}), as we wanted to prove.
\end{proof}

In Table \ref{TableDkWmedio}\ we collect some particular cases of (\ref%
{DkW_n,1/2}), obtained with the help of MATHEMATICA\ program.

\begin{center}
\begin{table}[htbp] \centering%
\caption{Derivative of $\mathrm{W}_{\kappa,\mu}$ with respect
to $\kappa$ by using (\ref{DkW_n,1/2}).}%
\begin{tabular}{|c|c|c|}
\hline
$\kappa $ & $\mu $ & $\frac{\partial \mathrm{W}_{\kappa ,\mu }\left(
x\right) }{\partial \kappa }$ \\ \hline\hline
$1$ & $\pm \frac{1}{2}$ & $e^{-x/2}\left( x\ln x-1\right) $ \\ \hline
$2$ & $\pm \frac{1}{2}$ & $e^{-x/2}\left[ x\left( x-2\right) \ln x-3x+1%
\right] $ \\ \hline
$3$ & $\pm \frac{1}{2}$ & $e^{-x/2}\left[ x\left( x^{2}-6x+6\right) \ln
x-5x^{2}+14x-2\right] $ \\ \hline
\end{tabular}%
\label{TableDkWmedio}%
\end{table}%
\end{center}

Note that for $n=0$, we obtain an indeterminate expression in (\ref%
{DkW_n,1/2}). We calculate this particular case with a result of the next
Section.

\begin{theorem}
The following reduction formula holds true:%
\begin{eqnarray}
&&\left. \frac{\partial \mathrm{W}_{\kappa ,\pm 1/2}\left( x\right) }{%
\partial \kappa }\right\vert _{\kappa =0}=e^{-x/2}  \label{DkW_0,1/2} \\
&&\left\{ \ln x+\frac{1}{4\sqrt{\pi }}\left[ G_{2,4}^{3,1}\left( \frac{x^{2}%
}{4}\left\vert
\begin{array}{c}
\frac{1}{2},1 \\
0,0,\frac{1}{2},-\frac{1}{2}%
\end{array}%
\right. \right) -\left( e^{x}-1\right) \,G_{1,3}^{3,0}\left( \frac{x^{2}}{4}%
\left\vert
\begin{array}{c}
1 \\
-\frac{1}{2},0,0%
\end{array}%
\right. \right) \right] \right\} ,  \notag
\end{eqnarray}%
where $G_{p,q}^{m,n}\left( z\left\vert
\begin{array}{c}
a_{1},\ldots ,a_{p} \\
b_{1},\ldots ,b_{q}%
\end{array}%
\right. \right) $ denotes the Meijer-G function.
\end{theorem}

\begin{proof}
According to \cite[Eqn. 13.18.2]{olver2010nist}, we have%
\begin{equation}
\mathrm{W}_{\kappa ,\kappa -1/2}\left( x\right) =e^{-x/2}x^{\kappa },
\label{W_k,k-1/2}
\end{equation}%
thus, performing the derivative with respect to $\kappa $,
\begin{equation*}
\left. \frac{\partial \mathrm{W}_{\kappa ,\mu }\left( x\right) }{\partial
\kappa }\right\vert _{\mu =\kappa -1/2}+\left. \frac{\partial \mathrm{W}%
_{\kappa ,\mu }\left( x\right) }{\partial \mu }\right\vert _{\mu =\kappa
-1/2}=e^{-x/2}x^{\kappa }\ln x.
\end{equation*}%
Taking $\kappa =0$ and considering (\ref{DkW_1/2=DkW_-1/2}), we have%
\begin{equation*}
\left. \frac{\partial \mathrm{W}_{\kappa ,\pm 1/2}\left( x\right) }{\partial
\kappa }\right\vert _{\kappa =0}=-\left. \frac{\partial \mathrm{W}_{0,\mu
}\left( x\right) }{\partial \mu }\right\vert _{\mu =-1/2}+e^{-x/2}\ln x.
\end{equation*}%
Finally, apply (\ref{DmW0}) and (\ref{DmuK_Meijer}), to arrive at (\ref%
{DkW_0,1/2})\ as we wanted to prove.
\end{proof}

\subsection{Derivative with respect to the second parameter $\partial
\mathrm{W}_{\protect\kappa ,\protect\mu }\left( x\right) /\partial \protect%
\mu $}

\begin{theorem}
For $2\mu \notin
\mathbb{Z}
$, the following parameter derivative formula of $\mathrm{W}_{\kappa ,\mu
}\left( x\right) $ holds true:%
\begin{eqnarray}
&&\left. \frac{\partial \mathrm{W}_{\kappa ,\pm \mu }\left( x\right) }{%
\partial \mu }\right\vert _{\kappa =\mu +1/2}=\pm \sqrt{x}e^{-x/2}
\label{DmuW_k=mu+1/2} \\
&&\left\{ x^{\mu }\left[ \frac{x}{2\mu +1}\,_{2}F_{2}\left( \left.
\begin{array}{c}
1,1 \\
2\mu +2,2%
\end{array}%
\right\vert x\right) -\psi \left( -2\mu \right) +\ln x\right] -\Gamma \left(
2\mu +1\right) x^{-\mu }\left( -x\right) ^{2\mu }\gamma \left( -2\mu
,-x\right) \right\} .  \notag
\end{eqnarray}
\end{theorem}

\begin{proof}
Differentiate the following reduction formula with respect to parameter $\mu
$ \cite[Eqn. 13.18.2]{olver2010nist}:%
\begin{equation*}
\mathrm{W}_{\mu +1/2,\pm \mu }\left( x\right) =e^{-x/2}x^{1/2+\mu },
\end{equation*}%
to obtain%
\begin{equation}
\left. \frac{\partial \mathrm{W}_{\kappa ,\pm \mu }\left( x\right) }{%
\partial \kappa }\right\vert _{\kappa =\mu +1/2}\pm \left. \frac{\partial
\mathrm{W}_{\kappa ,\pm \mu }\left( x\right) }{\partial \mu }\right\vert
_{\kappa =\mu +1/2}=e^{-x/2}x^{1/2+\mu }\ln x.  \label{Dmu(W_mu+1/2)}
\end{equation}%
Insert (\ref{DkW_k=mu+1/2})\ in (\ref{Dmu(W_mu+1/2)})\ to arrive at (\ref%
{DmuW_k=mu+1/2}), as we wanted to prove.
\end{proof}

Table \ref{Table_2A} shows the derivative of $\mathrm{W}_{\kappa ,\mu
}\left( x\right) $ with respect $\mu $ for particular values of $\kappa $
and $\mu $ using (\ref{DmuW_k=mu+1/2}) and the help of MATHEMATICA program.

\begin{center}
\begin{table}[tbp] \centering%
\caption{Derivative of $W_{\kappa,\mu}$ with respect
to $\mu$ by using (\ref{DmuW_k=mu+1/2}).}%
\rotatebox{90}{
\begin{tabular}{|c|c|c|}
\hline
$\kappa $ & $\mu $ & $\frac{\partial \mathrm{W}_{\kappa ,\mu }\left(
x\right) }{\partial \mu }$ \\ \hline\hline
$-\frac{3}{4}$ & $\pm \frac{5}{4}$ & $\pm \frac{1}{3}x^{-3/4}e^{-x/2}\,\left[
2x\,_{2}F_{2}\left( 1,1;-\frac{1}{2},2;x\right) +3\pi \,\mathrm{erfi}\left(
\sqrt{x}\right) +2\sqrt{\pi \,x}\,e^{x}\left( 2x-3\right) -3\gamma +8-3\ln
\left( 4x\right) \right] $ \\ \hline
$-\frac{1}{4}$ & $\pm \frac{3}{4}$ & $\pm x^{-1/4}e^{-x/2}\left[
2x\,_{2}F_{2}\left( 1,1;\frac{1}{2},2;x\right) +\pi \,\mathrm{erfi}\left(
\sqrt{x}\right) -2\sqrt{\pi \,x}\,e^{x}-\gamma +2-\ln \left( 4x\right) %
\right] $ \\ \hline
$-\frac{1}{6}$ & $\pm \frac{2}{3}$ & $%
\begin{array}{l}
\pm \frac{1}{6}x^{-5/6}e^{-x/2}\left\{ 3x^{2/3}\left[ 6x\,_{2}F_{2}\left(
1,1;\frac{2}{3},2;x\right) -2\gamma +6-3\ln 3-2\ln x\right] \right.  \\
\quad \left. -6x^{2}\Gamma \left( -\frac{1}{3}\right) \mathrm{E}%
_{-1/3}\left( -x\right) -\sqrt{3}\pi \left[ x^{2/3}+4\left( -x\right) ^{2/3}%
\right] \right\}
\end{array}%
$ \\ \hline
$\frac{1}{6}$ & $\pm \frac{1}{3}$ & $%
\begin{array}{l}
\pm \frac{1}{6}x^{-1/6}e^{-x/2}\left\{ -3x^{1/3}\left[ 6x\,_{2}F_{2}\left(
1,1;\frac{4}{3},2;x\right) +2\gamma +3\ln 3+2\ln x\right] \right.  \\
\quad \left. -6x\,\Gamma \left( \frac{1}{3}\right) \mathrm{E}_{1/3}\left(
-x\right) +\sqrt{3}\pi \left[ x^{1/3}-4\left( -x\right) ^{1/3}\right]
\right\}
\end{array}%
$ \\ \hline
$\frac{1}{4}$ & $\pm \frac{1}{4}$ & $\pm x^{1/4}e^{-x/2}\left[
-2x\,_{2}F_{2}\left( 1,1;\frac{3}{2},2;x\right) +\pi \,\mathrm{erfi}\left(
\sqrt{x}\right) -\gamma -\ln \left( 4x\right) \right] $ \\ \hline
$\frac{3}{4}$ & $\pm \frac{1}{4}$ & $\pm \frac{1}{3}x^{1/4}e^{-x/2}\left\{
\sqrt{x}\left[ 2x\,_{2}F_{2}\left( 1,1;\frac{5}{2},2;x\right) -3\left( \pi \,%
\mathrm{erfi}\left( \sqrt{x}\right) -\gamma +2-\ln \left( 4x\right) \right) %
\right] +3\sqrt{\pi }\,e^{x}\right\} $ \\ \hline
$\frac{5}{6}$ & $\pm \frac{1}{3}$ & $%
\begin{array}{l}
\pm \frac{1}{30}x^{1/6}e^{-x/2}\left\{ 18\,x^{5/3}\,_{2}F_{2}\left( 1,1;%
\frac{8}{3},2;x\right) +15\,x^{2/3}\left( 2\gamma +3\ln 3+2\ln x-3\right)
\right.  \\
\quad +\left. 30\,\Gamma \left( \frac{5}{3}\right) \mathrm{E}_{5/3}\left(
-x\right) +5\sqrt{3}\pi \left[ x^{2/3}+4\left( -x\right) ^{1/3}\right]
\right\}
\end{array}%
$ \\ \hline
$\frac{5}{4}$ & $\pm \frac{3}{4}$ & $\pm \frac{1}{30}x^{-1/4}e^{-x/2}\left\{
2x^{3/2}\left[ 6x\,_{2}F_{2}\left( 1,1;\frac{7}{2},2;x\right) -5\left( \pi \,%
\mathrm{erfi}\left( \sqrt{x}\right) -3\gamma +8-3\ln \left( 4x\right)
\right) \right] +15\sqrt{\pi }\,e^{x}\left( 2x+1\right) \right\} $ \\ \hline
\end{tabular}%
}
\label{Table_2A}%
\end{table}%
\end{center}

\begin{theorem}
The following parameter derivative formula of $\mathrm{W}_{\kappa ,\mu
}\left( x\right) $ holds true:%
\begin{equation}
\frac{\partial \mathrm{W}_{0,\mu }\left( x\right) }{\partial \mu }=\mathrm{%
sgn}\left( \mu \right) \sqrt{\frac{x}{\pi }}\left. \frac{\partial K_{\mu
}\left( x/2\right) }{\partial \mu }\right\vert _{\left\vert \mu \right\vert
},  \label{DmW0}
\end{equation}%
where $K_{\nu }\left( x\right) $ denotes the modified Bessel of the second
kind (Macdonald function).
\end{theorem}

\begin{proof}
Differentiate with respect to $\mu $ the expression \cite[Eqn. 13.18.9]%
{olver2010nist}:
\begin{equation}
\mathrm{W}_{0,\mu }\left( x\right) =\sqrt{\frac{x}{\pi }}K_{\mu }\left(
\frac{x}{2}\right) ,  \label{W_0_mu}
\end{equation}%
to obtain
\begin{equation*}
\left. \frac{\partial \mathrm{W}_{0,\pm \mu }\left( x\right) }{\partial \mu }%
\right\vert _{\mu \geq 0}=\pm \left. \frac{\partial \mathrm{W}_{0,\mu
}\left( x\right) }{\partial \mu }\right\vert _{\mu \geq 0}=\pm \sqrt{\frac{x%
}{\pi }}\left. \frac{\partial K_{\mu }\left( x/2\right) }{\partial \mu }%
\right\vert _{\mu \geq 0},
\end{equation*}%
as we wanted to prove.
\end{proof}

The order derivative of $K_{\mu }\left( x\right) $ is given in terms of
Meijer-G functions for$\,\mathrm{Re\,}x>0$, and $\mu \geq 0$ \cite%
{gonzalez2018closed}:%
\begin{eqnarray}
&&\frac{\partial K_{\mu }\left( x\right) }{\partial \mu }
\label{DmuK_Meijer} \\
&=&\frac{\mu }{2}\left[ \frac{\,K_{\mu }\left( x\right) }{\sqrt{\pi }}%
G_{2,4}^{3,1}\left( x^{2}\left\vert
\begin{array}{c}
\frac{1}{2},1 \\
0,0,\mu ,-\mu%
\end{array}%
\right. \right) -\sqrt{\pi }I_{\mu }\left( x\right) \,G_{2,4}^{4,0}\left(
x^{2}\left\vert
\begin{array}{c}
\frac{1}{2},1 \\
0,0,\mu ,-\mu%
\end{array}%
\right. \right) \right] ,  \notag
\end{eqnarray}%
where $I_{\nu }\left( x\right) $ is the \textit{modified Bessel function};\
or in terms of generalized hypergeometric functions for$\,\mathrm{Re\,}x>0$,
$\mu >0$, and $2\mu \notin
\mathbb{Z}
$ \cite{brychkov2016higher}:%
\begin{eqnarray}
&&\frac{\partial K_{\mu }\left( x\right) }{\partial \mu }  \label{DmuK_Hyper}
\\
&=&\frac{\pi }{2}\csc \left( \pi \mu \right) \left\{ \pi \cot \left( \pi \mu
\right) \,I_{\mu }\left( x\right) -\left[ I_{\mu }\left( x\right) +I_{-\mu
}\left( x\right) \right]
\begin{array}{c}
\displaystyle
\\
\displaystyle%
\end{array}%
\right.  \notag \\
&&\left. \left[ \frac{x^{2}}{4\left( 1-\mu ^{2}\right) }\,_{3}F_{4}\left(
\left.
\begin{array}{c}
1,1,\frac{3}{2} \\
2,2,2-\mu ,2+\mu%
\end{array}%
\right\vert x^{2}\right) +\ln \left( \frac{x}{2}\right) -\psi \left( \mu
\right) -\frac{1}{2\mu }\right] \right\}  \notag \\
&&+\frac{1}{4}\left\{ I_{-\mu }\left( x\right) \Gamma ^{2}\left( -\mu
\right) \left( \frac{x}{2}\right) ^{2\mu }\,_{2}F_{3}\left( \left.
\begin{array}{c}
\mu ,\frac{1}{2}+\mu \\
1+\mu ,1+\mu ,1+2\mu%
\end{array}%
\right\vert x^{2}\right) \right.  \notag \\
&&\quad -\left. I_{\mu }\left( x\right) \Gamma ^{2}\left( \mu \right) \left(
\frac{x}{2}\right) ^{-2\mu }\,_{2}F_{3}\left( \left.
\begin{array}{c}
-\mu ,\frac{1}{2}-\mu \\
1-\mu ,1-\mu ,1-2\mu%
\end{array}%
\right\vert x^{2}\right) \right\} .  \notag
\end{eqnarray}

There are different expressions for the order derivatives of the Bessel
functions \cite{apelblat1985integral,brychkov2005derivatives}. This subject
is summarized in \cite{apelblat2020bessel}, where general results are
presented in terms of convolution integrals, and order derivatives of Bessel
functions are found for particular values of the order.

Using (\ref{DmW0}), (\ref{DmuK_Meijer})\ and (\ref{DmuK_Hyper}), some
derivatives of $\mathrm{W}_{\kappa ,\mu }\left( x\right) $ with respect $\mu
$ has been calculated with the help of MATHEMATICA program, and they are
presented in Table \ref{Table_2}.

\begin{center}
\begin{table}[tbp] \centering%
\caption{Derivative of $\mathrm{W}_{\kappa,\mu}$ with respect
to $\mu$ by using (\ref{DmW0}).}%
\rotatebox{90}{
\begin{tabular}{|c|c|c|}
\hline
$\kappa $ & $\mu $ & $\frac{\partial \mathrm{W}_{\kappa ,\mu }\left(
x\right) }{\partial \mu }$ \\ \hline\hline
$0$ & $0$ & $0$ \\ \hline
$0$ & $\pm \frac{1}{4}$ & $%
\begin{array}{l}
\pm \frac{1}{8\sqrt{\pi }}\left\{ 4\pi \sqrt{2\,x}\left( \pi \,I_{1/4}\left(
\frac{x}{2}\right) -\left[ I_{1/4}\left( \frac{x}{2}\right) +I_{-1/4}\left(
\frac{x}{2}\right) \right] \left[ \frac{x^{2}}{15}\,_{3}F_{4}\left( \left.
\begin{array}{c}
1,1,\frac{3}{2} \\
2,2,\frac{7}{4},\frac{9}{4}%
\end{array}%
\right\vert \frac{x^{2}}{4}\right) +\ln \left( \frac{x}{4}\right) -\psi
\left( \frac{1}{4}\right) -2\right] \right) \right. \\
\quad -\left. 4\,\Gamma ^{2}\left( \frac{1}{4}\right) I_{1/4}\left( \frac{x}{%
2}\right) \,_{2}F_{3}\left( \left.
\begin{array}{c}
-\frac{1}{4},\frac{1}{4} \\
\frac{3}{4},\frac{3}{4},\frac{1}{2}%
\end{array}%
\right\vert \frac{x^{2}}{4}\right) +x\,\Gamma ^{2}\left( -\frac{1}{4}\right)
\,I_{-1/4}\left( \frac{x}{2}\right) \,_{2}F_{3}\left( \left.
\begin{array}{c}
\frac{1}{4},\frac{3}{4} \\
\frac{5}{4},\frac{5}{4},\frac{3}{2}%
\end{array}%
\right\vert \frac{x^{2}}{4}\right) \right\}%
\end{array}%
$ \\ \hline
$0$ & $\pm \frac{1}{3}$ & $%
\begin{array}{l}
\pm \frac{x^{-1/6}}{384\sqrt{\pi }}\left\{ \pi \,x^{2/3}\left( 128\pi
\,I_{1/3}\left( \frac{x}{2}\right) -\sqrt{3}\left[ I_{1/4}\left( \frac{x}{2}%
\right) +I_{-1/4}\left( \frac{x}{2}\right) \right] \left[ 9\,x^{2}%
\,_{3}F_{4}\left( \left.
\begin{array}{c}
1,1,\frac{3}{2} \\
2,2,\frac{5}{3},\frac{7}{3}%
\end{array}%
\right\vert \frac{x^{2}}{4}\right) +64\left( 2\ln \left( \frac{x}{4}\right)
-2\psi \left( \frac{1}{3}\right) -3\right) \right] \right) \right. \\
\quad -\left. 48\sqrt[3]{2}\left[ 3x\,\Gamma \left( -\frac{1}{3}\right)
\,_{0}F_{1}\left( \left.
\begin{array}{c}
- \\
\frac{2}{3}%
\end{array}%
\right\vert \frac{x^{2}}{16}\right) \,_{2}F_{3}\left( \left.
\begin{array}{c}
\frac{1}{3},\frac{5}{6} \\
\frac{4}{3},\frac{4}{3},\frac{5}{3}%
\end{array}%
\right\vert \frac{x^{2}}{4}\right) +\Gamma ^{2}\left( \frac{1}{3}\right)
I_{1/3}\left( \frac{x}{2}\right) \,_{2}F_{3}\left( \left.
\begin{array}{c}
-\frac{1}{3},\frac{1}{6} \\
\frac{1}{3},\frac{2}{3},\frac{2}{3}%
\end{array}%
\right\vert \frac{x^{2}}{4}\right) \right] \right\}%
\end{array}%
$ \\ \hline
$0$ & $\pm \frac{1}{2}$ & $\pm \frac{1}{4\sqrt{\pi }}e^{-x/2}\left[
G_{2,4}^{3,1}\left( \frac{x^{2}}{4}\left\vert
\begin{array}{c}
\frac{1}{2},1 \\
0,0,\frac{1}{2},-\frac{1}{2}%
\end{array}%
\right. \right) -\left( e^{x}-1\right) \,G_{1,3}^{3,0}\left( \frac{x^{2}}{4}%
\left\vert
\begin{array}{c}
1 \\
-\frac{1}{2},0,0%
\end{array}%
\right. \right) \right] $ \\ \hline
$0$ & $\pm \frac{2}{3}$ & $%
\begin{array}{l}
\pm \frac{1}{\sqrt{\pi }}\sqrt{x}\left\{ -\frac{1}{3}\pi ^{2}\,I_{2/3}\left(
\frac{x}{2}\right) -\frac{\pi }{\sqrt{3}}\left[ I_{-2/3}\left( \frac{x}{2}%
\right) +I_{2/3}\left( \frac{x}{2}\right) \right] \left[ \frac{9}{80}%
x^{2}\,_{3}F_{4}\left( \left.
\begin{array}{c}
1,1,\frac{3}{2} \\
2,2,\frac{4}{3},\frac{8}{3}%
\end{array}%
\right\vert \frac{x^{2}}{4}\right) +\ln \left( \frac{x}{4}\right) -\psi
\left( \frac{2}{3}\right) -\frac{3}{4}\right] \right. \\
\quad +\left. 2^{-14/3}x^{4/3}\,\Gamma ^{2}\left( -\frac{2}{3}\right)
I_{-2/3}\left( \frac{x}{2}\right) \,_{2}F_{3}\left( \left.
\begin{array}{c}
\frac{2}{3},\frac{7}{6} \\
\frac{5}{3},\frac{5}{3},\frac{7}{3}%
\end{array}%
\right\vert \frac{x^{2}}{4}\right) -2^{2/3}x^{-4/3}\,\Gamma ^{2}\left( \frac{%
2}{3}\right) \,I_{2/3}\left( \frac{x}{2}\right) \,_{2}F_{3}\left( \left.
\begin{array}{c}
-\frac{2}{3},-\frac{1}{6} \\
-\frac{1}{3},\frac{1}{3},\frac{1}{3}%
\end{array}%
\right\vert \frac{x^{2}}{4}\right) \right\}%
\end{array}%
$ \\ \hline
$0$ & $\pm \frac{3}{4}$ & $%
\begin{array}{l}
\pm \frac{1}{672\sqrt{\pi }x}\left\{ x^{3/2}\left( -8\sqrt{2}\pi \left[
I_{-3/4}\left( \frac{x}{2}\right) +I_{3/4}\left( \frac{x}{2}\right) \right] %
\left[ 6\,x^{2}\,_{3}F_{4}\left( \left.
\begin{array}{c}
1,1,\frac{3}{2} \\
2,2,\frac{5}{4},\frac{11}{4}%
\end{array}%
\right\vert \frac{x^{2}}{4}\right) +42\left[ \ln \left( 2x\right) +\gamma %
\right] -28\right] \right. \right. \\
+\left. \left. \,21\,x^{3/2}\,\Gamma ^{2}\left( -\frac{3}{4}\right)
I_{-3/4}\left( \frac{x}{2}\right) \,_{2}F_{3}\left( \left.
\begin{array}{c}
\frac{3}{4},\frac{5}{4} \\
\frac{7}{4},\frac{7}{4},\frac{5}{2}%
\end{array}%
\right\vert \frac{x^{2}}{4}\right) +336\pi \,K_{3/4}\left( \frac{x}{2}%
\right) \right) -1344\,\,\Gamma ^{2}\left( \frac{3}{4}\right)
\,I_{3/4}\left( \frac{x}{2}\right) \,_{2}F_{3}\left( \left.
\begin{array}{c}
-\frac{3}{4},-\frac{1}{4} \\
-\frac{1}{2},\frac{1}{4},\frac{1}{4}%
\end{array}%
\right\vert \frac{x^{2}}{4}\right) \right\}%
\end{array}%
$ \\ \hline
$0$ & $\pm 1$ & $\pm \frac{1}{2\pi }\sqrt{x}\left[ K_{1}\left( \frac{x}{2}%
\right) G_{1,3}^{2,1}\left( \frac{x^{2}}{4}\left\vert
\begin{array}{c}
\frac{1}{2} \\
0,0,-1%
\end{array}%
\right. \right) -\pi \,I_{1}\left( \frac{x}{2}\right) \,G_{1,3}^{3,0}\left(
\frac{x^{2}}{4}\left\vert
\begin{array}{c}
\frac{1}{2} \\
-1,0,0%
\end{array}%
\right. \right) \right] $ \\ \hline
$0$ & $\pm \frac{3}{2}$ & $\pm \frac{1}{4\sqrt{\pi }x}e^{-x/2}\left[ 3\left(
x+2\right) G_{2,4}^{3,1}\left( \frac{x^{2}}{4}\left\vert
\begin{array}{c}
\frac{1}{2},1 \\
0,0,\frac{3}{2},-\frac{3}{2}%
\end{array}%
\right. \right) -3\left[ e^{x}\left( x-2\right) +x+2\right]
\,G_{2,4}^{4,0}\left( \frac{x^{2}}{4}\left\vert
\begin{array}{c}
\frac{1}{2},1 \\
-\frac{3}{2},0,0,\frac{3}{2}%
\end{array}%
\right. \right) \right] $ \\ \hline
$0$ & $\pm 2$ & $\pm \frac{1}{\pi }\sqrt{x}\left[ K_{2}\left( \frac{x}{2}%
\right) G_{2,4}^{3,1}\left( \frac{x^{2}}{4}\left\vert
\begin{array}{c}
\frac{1}{2},1 \\
0,0,2,-2%
\end{array}%
\right. \right) -\pi \,I_{2}\left( \frac{x}{2}\right) \,G_{2,4}^{4,0}\left(
\frac{x^{2}}{4}\left\vert
\begin{array}{c}
\frac{1}{2},1 \\
-2,0,0,2%
\end{array}%
\right. \right) \right] $ \\ \hline
\end{tabular}%
}
\label{Table_2}%
\end{table}%
\end{center}

\section{Parameter differentiation of $\mathrm{W}_{\protect\kappa ,\protect%
\mu }$ via integral representations\label{Section: Integral representations}}

\subsection{Derivative with respect to the first parameter $\partial \mathrm{%
W}_{\protect\kappa ,\protect\mu }\left( x\right) /\partial \protect\kappa $}

Integral representations of the Whittaker function $\mathrm{W}_{\kappa ,\mu
}\left( z\right) $ for $\mathrm{Re}\left( \mu -\kappa \right) >-\frac{1}{2}$
and $\left\vert \mathrm{arg\,}z\right\vert <\frac{\pi }{2}$ are given in the
form of Laplace transform \cite[Sect. 7.4.2]{magnus2013formulas}:
\begin{eqnarray}
&&\mathrm{W}_{\kappa ,\mu }\left( z\right)  \label{W_k,mu_int_1} \\
&=&\frac{\,z^{\mu +1/2}e^{-z/2}}{\Gamma \left( \mu -\kappa +\frac{1}{2}%
\right) }\int_{0}^{\infty }e^{-z\,t}t^{\mu -\kappa -1/2}\left( 1+t\right)
^{\mu +\kappa -1/2}dt,  \notag
\end{eqnarray}%
and as the infinite integral:%
\begin{eqnarray}
&&\mathrm{W}_{\kappa ,\mu }\left( z\right)  \label{W_k,mu_int_2} \\
&=&\frac{\,z^{\mu +1/2}e^{z/2}}{\Gamma \left( \mu -\kappa +\frac{1}{2}%
\right) }\int_{1}^{\infty }e^{-z\,t}t^{\mu +\kappa -1/2}\left( t-1\right)
^{\mu -\kappa -1/2}dt.  \notag
\end{eqnarray}%
In order to calculate the first derivative of $\mathrm{W}_{\kappa ,\mu
}\left( x\right) $\ with respect to parameter $\kappa $, let us introduce
the following finite logarithmic integrals.

\begin{definition}
For $\mathrm{Re}\left( \mu -\kappa \right) >-\frac{1}{2}$ and $x>0$, define:%
\begin{eqnarray}
I_{1}^{\ast }\left( \kappa ,\mu ;x\right) &=&\int_{0}^{\infty }e^{-xt}t^{\mu
-\kappa -1/2}\left( 1+t\right) ^{\mu +\kappa -1/2}\ln \left( \frac{1+t}{t}%
\right) dt,  \label{I*1_def} \\
I_{2}^{\ast }\left( \kappa ,\mu ;x\right) &=&\int_{1}^{\infty }e^{-xt}t^{\mu
+\kappa -1/2}\left( t-1\right) ^{\mu -\kappa -1/2}\ln \left( \frac{t}{t-1}%
\right) dt.  \label{I*2_def}
\end{eqnarray}
\end{definition}

For $x>0$, differentiation of (\ref{W_k,mu_int_1}) and (\ref{W_k,mu_int_2})\
with respect to parameter $\kappa $ yields respectively%
\begin{eqnarray}
&&\frac{\partial \mathrm{W}_{\kappa ,\mu }\left( x\right) }{\partial \kappa }
\notag \\
&=&\psi \left( \mu -\kappa +\frac{1}{2}\right) \mathrm{W}_{\kappa ,\mu
}\left( x\right) +\frac{\,x^{\mu +1/2}e^{-x/2}}{\Gamma \left( \mu -\kappa +%
\frac{1}{2}\right) }I_{1}^{\ast }\left( \kappa ,\mu ;x\right)  \label{DkW_I1}
\\
&=&\psi \left( \mu -\kappa +\frac{1}{2}\right) \mathrm{W}_{\kappa ,\mu
}\left( x\right) +\frac{\,x^{\mu +1/2}e^{x/2}}{\Gamma \left( \mu -\kappa +%
\frac{1}{2}\right) }I_{2}^{\ast }\left( \kappa ,\mu ;x\right) .
\label{DkW_I2}
\end{eqnarray}

Note that, from (\ref{DkW_I1}) and (\ref{DkW_I2}), we have%
\begin{equation}
I_{2}^{\ast }\left( \kappa ,\mu ;x\right) =e^{-x}I_{1}^{\ast }\left( \kappa
,\mu ;x\right) .  \label{I1->I2}
\end{equation}

\begin{theorem}
The following integral holds true for $\frac{1}{2}+\mu -\kappa >0$ and $x>0$:%
\begin{eqnarray}
&&I_{1}^{\ast }\left( \kappa ,\mu ;x\right)  \label{I1s_general} \\
&=&\mathrm{B}\left( \frac{1}{2}+\mu -\kappa ,-2\mu \right)  \notag \\
&&\left\{ \left[ \psi \left( \frac{1}{2}-\mu -\kappa \right) -\psi \left(
\frac{1}{2}+\mu -\kappa \right) \right] \,_{1}F_{1}\left( \left.
\begin{array}{c}
\frac{1}{2}+\mu -\kappa \\
1+2\mu%
\end{array}%
\right\vert x\right) \right.  \notag \\
&&-\left. G^{\left( 1\right) }\left( \left.
\begin{array}{c}
\frac{1}{2}+\mu -\kappa \\
1+2\mu%
\end{array}%
\right\vert x\right) \right\} -\Gamma \left( 2\mu \right) \,x^{-2\mu
}\,G^{\left( 1\right) }\left( \left.
\begin{array}{c}
\frac{1}{2}-\mu -\kappa \\
1-2\mu%
\end{array}%
\right\vert x\right) ,  \notag
\end{eqnarray}%
where $\mathrm{B}\left( x,y\right) =\frac{\Gamma \left( x\right) \Gamma
\left( y\right) }{\Gamma \left( x+y\right) }$ denotes the beta function.
\end{theorem}

\begin{proof}
Compare (\ref{DkW_(1)})\ to (\ref{DkW_I1})\ and take into account (\ref%
{M_k,mu_def})\ to arrive at (\ref{I1s_general}), as we wanted to prove.
\end{proof}

Now, we derive a Lemma that will be applied throughout this Section and the
next one.

\begin{lemma}
For $\nu \geq 0$ and $x>0$, the following Laplace transform holds true:%
\begin{eqnarray}
&&\mathcal{I}_{\pm }\left( \nu ,x\right)  \label{I(nu,x)_resultado} \\
&=&\int_{0}^{\infty }e^{-xt}t^{\nu }\ln \left( t^{\pm 1}\left( 1+t\right)
\right) dt  \notag \\
&=&\frac{\Gamma \left( \nu +1\right) }{x^{\nu +1}}\left\{ \frac{x}{\nu +1}%
\,_{2}F_{2}\left( \left.
\begin{array}{c}
1,1 \\
2,2+\nu%
\end{array}%
\right\vert -x\right) \right.  \notag \\
&&-\left. e^{-i\pi \nu }\,\Gamma \left( -\nu ,x\right) \,\gamma \left( \nu
+1,-x\right) +\left( 1\pm 1\right) \left[ \psi \left( \nu +1\right) -\ln x%
\right] \right\} ,  \notag
\end{eqnarray}%
where $\Gamma \left( \nu ,z\right) $ and $\gamma \left( \nu ,z\right) $
denotes respectively the upper and lower incomplete gamma functions, (\ref%
{gamma_def})\ and (\ref{Gamma_def}).
\end{lemma}

\begin{proof}
Split the integral in two terms as follows:%
\begin{equation*}
\mathcal{I}_{\pm }\left( \nu ,x\right) =\underset{\mathcal{I}_{a}\left( \nu
,x\right) }{\underbrace{\int_{0}^{\infty }e^{-xt}t^{\nu }\ln \left(
1+t\right) dt}}\pm \underset{\mathcal{I}_{b}\left( \nu ,x\right) }{%
\underbrace{\int_{0}^{\infty }e^{-xt}t^{\nu }\ln t\,dt}},
\end{equation*}%
and apply the Laplace transform for $x>0$ \cite[Eqn. 2.5.2(4)]%
{prudnikov1986integrals}\footnote{%
It is worth noting that there is an incorrect sign in the reference cited.}:%
\begin{eqnarray*}
&&\int_{0}^{\infty }e^{-xt}t^{\nu }\ln \left( at+b\right) dt \\
&=&-\frac{\pi }{\left( \nu +1\right) \sin \pi \nu }\left( \frac{b}{a}\right)
^{\nu +1}\,_{1}F_{1}\left( \left.
\begin{array}{c}
\nu +1 \\
\nu +2%
\end{array}%
\right\vert \frac{b\,x}{a}\right) \\
&&+\frac{\Gamma \left( \nu +1\right) }{x^{\nu +1}}\left[ \psi \left( \nu
+1\right) -\ln \left( \frac{x}{a}\right) +\frac{b\,x}{a\,\nu }%
\,_{2}F_{2}\left( \left.
\begin{array}{c}
1,1 \\
2,1-\nu%
\end{array}%
\right\vert \frac{b\,x}{a}\right) \right] ,
\end{eqnarray*}%
to obtain

\begin{eqnarray}
&&\mathcal{I}_{a}\left( \nu ,x\right)  \label{I*1a_(1)} \\
&=&-\frac{\pi }{\left( \nu +1\right) \sin \pi \nu }\,_{1}F_{1}\left( \left.
\begin{array}{c}
\nu +1 \\
\nu +2%
\end{array}%
\right\vert x\right)  \notag \\
&&+\frac{\Gamma \left( \nu +1\right) }{x^{\nu +1}}\left[ \psi \left( \nu
+1\right) -\ln x+\frac{x}{\nu }\,_{2}F_{2}\left( \left.
\begin{array}{c}
1,1 \\
2,1-\nu%
\end{array}%
\right\vert x\right) \right] ,  \notag
\end{eqnarray}%
and%
\begin{equation}
\mathcal{I}_{b}\left( \nu ,x\right) =\frac{\Gamma \left( \nu +1\right) }{%
x^{\nu +1}}\left[ \psi \left( \nu +1\right) -\ln x\right] .
\label{I*1b_resultado}
\end{equation}%
Note that, according to Kummer's transformation (\ref{Kummer_transform}),\
and to the reduction formula \cite[Eqn. 7.11.1(14)]{prudnikov1986integrals}:%
\begin{equation*}
_{1}F_{1}\left( \left.
\begin{array}{c}
1 \\
b%
\end{array}%
\right\vert z\right) =\left( b-1\right) z^{1-b}e^{z}\gamma \left(
1-b,z\right) ,
\end{equation*}%
we have for $x>0$%
\begin{eqnarray}
_{1}F_{1}\left( \left.
\begin{array}{c}
a \\
a+1%
\end{array}%
\right\vert x\right) &=&e^{x}\,_{1}F_{1}\left( \left.
\begin{array}{c}
1 \\
a+1%
\end{array}%
\right\vert -x\right)  \label{1F1_reduction_gamma} \\
&=&a\left( -x\right) ^{-a}\gamma \left( a,-x\right)  \notag \\
&=&a\,e^{-i\pi a}x^{-a}\gamma \left( a,-x\right) ,  \notag
\end{eqnarray}%
thus (\ref{I*1a_(1)})\ becomes%
\begin{eqnarray}
&&\mathcal{I}_{a}\left( \nu ,x\right)  \label{I*1a_resultado} \\
&=&\frac{1}{x^{\nu +1}}\left\{ \frac{\pi }{\sin \pi \nu }e^{-i\pi \nu
}\gamma \left( \nu +1,-x\right) \right.  \notag \\
&&+\left. \Gamma \left( \nu +1\right) \left[ \psi \left( \nu +1\right) -\ln
x+\frac{x}{\nu }\,_{2}F_{2}\left( \left.
\begin{array}{c}
1,1 \\
2,1-\nu%
\end{array}%
\right\vert x\right) \right] \right\} .  \notag
\end{eqnarray}%
Now, insert (\ref{I*1b_resultado})\ and (\ref{I*1a_resultado})\ in (\ref%
{I*1_(1)})\ to arrive at
\begin{eqnarray}
&&\mathcal{I}_{\pm }\left( \nu ,x\right) =\frac{1}{x^{\nu +1}}
\label{I(nu,x)_2} \\
&&\left\{ \frac{\pi }{\sin \pi \nu }e^{-i\pi \nu }\gamma \left( \nu
+1,-x\right) +x\,\Gamma \left( \nu \right) \,_{2}F_{2}\left( \left.
\begin{array}{c}
1,1 \\
2,1-\nu%
\end{array}%
\right\vert x\right) \right\}  \notag \\
&&+\left( 1\pm 1\right) \frac{\Gamma \left( \nu +1\right) }{x^{\nu +1}}\left[
\psi \left( \nu +1\right) -\ln x\right] .  \notag
\end{eqnarray}%
Next, apply the transformation formula \cite[Eqn. 7.12.1(7)]%
{prudnikov1986integrals}:%
\begin{eqnarray*}
&&_{2}F_{2}\left( \left.
\begin{array}{c}
1,a \\
a+1,b%
\end{array}%
\right\vert z\right) +\frac{b-1}{a-b+1}\,_{2}F_{2}\left( \left.
\begin{array}{c}
1,a \\
a+1,2+a-b%
\end{array}%
\right\vert -z\right) \\
&=&\frac{a}{a-b+1}\,_{1}F_{1}\left( \left.
\begin{array}{c}
a-b+1 \\
a-b+2%
\end{array}%
\right\vert z\right) \,_{1}F_{1}\left( \left.
\begin{array}{c}
b-1 \\
b%
\end{array}%
\right\vert -z\right) ,
\end{eqnarray*}%
taking $a=1$ and $b=1-\nu $, and applying again (\ref{1F1_reduction_gamma}),
to arrive at%
\begin{eqnarray}
&&_{2}F_{2}\left( \left.
\begin{array}{c}
1,1 \\
2,1-\nu%
\end{array}%
\right\vert x\right)  \label{2F2_transform} \\
&=&\nu \left\{ \frac{1}{\nu +1}\,_{2}F_{2}\left( \left.
\begin{array}{c}
1,1 \\
2,2+\nu%
\end{array}%
\right\vert -x\right) +\frac{e^{-i\pi \nu }}{x}\gamma \left( -\nu ,x\right)
\,\gamma \left( \nu +1,-x\right) \right\} .  \notag
\end{eqnarray}%
Insert (\ref{2F2_transform})\ in (\ref{I(nu,x)_2})\ to get%
\begin{eqnarray}
&&\mathcal{I}_{\pm }\left( \nu ,x\right)  \label{I(nu,x)_(1)} \\
&=&\frac{1}{x^{\nu +1}}\left\{ e^{-i\pi \nu }\gamma \left( \nu +1,-x\right) %
\left[ \frac{\pi }{\sin \pi \nu }+\Gamma \left( \nu +1\right) \gamma \left(
-\nu ,x\right) \right] \right.  \notag \\
&&+\left. \Gamma \left( \nu +1\right) \left[ \frac{x\,}{\nu +1}%
\,_{2}F_{2}\left( \left.
\begin{array}{c}
1,1 \\
2,1-\nu%
\end{array}%
\right\vert x\right) +\left( 1\pm 1\right) \left[ \psi \left( \nu +1\right)
-\ln x\right] \right] \right\} .  \notag
\end{eqnarray}%
Applying the properties \cite[Eqn. 45:0:1]{oldham2009atlas}
\begin{equation}
\Gamma \left( \nu \right) =\gamma \left( \nu ,z\right) +\Gamma \left( \nu
,z\right) ,  \label{Gamma=Lower+Upper}
\end{equation}%
and \cite[Eqn. 1.2.2]{lebedev1965special}
\begin{equation*}
\Gamma \left( z\right) \Gamma \left( 1-z\right) =\pi \csc \pi z,
\end{equation*}%
rewrite (\ref{I(nu,x)_(1)})\ as (\ref{I(nu,x)_resultado}), as we wanted to
prove.
\end{proof}

\begin{theorem}
The following integral holds true for $\mu >0$ and $x>0$:%
\begin{eqnarray}
&&I_{1}^{\ast }\left( \frac{1}{2}-\mu ,\mu ;x\right)  \label{I*1_reduction}
\\
&=&\mathcal{I}_{-}\left( 2\mu -1,x\right)  \label{I(2*mu-1,x)} \\
&=&\frac{\Gamma \left( 2\mu \right) }{x^{2\mu }}\left\{ \frac{\,x}{2\mu }%
\,_{2}F_{2}\left( \left.
\begin{array}{c}
1,1 \\
2,1+2\mu%
\end{array}%
\right\vert -x\right) +e^{-2\pi i\mu }\,\Gamma \left( 1-2\mu ,x\right)
\,\gamma \left( 2\mu ,-x\right) \right\} .  \notag
\end{eqnarray}
\end{theorem}

\begin{proof}
From (\ref{I*1_def}) and (\ref{I(nu,x)_resultado}), we obtain the desired
result.
\end{proof}

\begin{remark}
If we insert (\ref{I(nu,x)_2})\ in (\ref{I(2*mu-1,x)}), we obtain the
following alternative form:%
\begin{eqnarray}
&&I_{1}^{\ast }\left( \frac{1}{2}-\mu ,\mu ;x\right)
\label{I*1_reduction_bis} \\
&=&\frac{1}{x^{2\mu }}\left\{ \pi \left[ \cot \left( 2\pi \mu \right) -i%
\right] \,\gamma \left( 2\mu ,-x\right) +\,x\,\Gamma \left( 2\mu -1\right)
\,_{2}F_{2}\left( \left.
\begin{array}{c}
1,1 \\
2,2-2\mu%
\end{array}%
\right\vert x\right) \right\} .  \notag
\end{eqnarray}
\end{remark}

\begin{theorem}
The following reduction formula holds true for $-2\mu \neq 0,1,\ldots $ and $%
x>0$:%
\begin{eqnarray}
&&\left. \frac{\partial \mathrm{W}_{\kappa ,\mu }\left( x\right) }{\partial
\kappa }\right\vert _{\kappa =-\mu +1/2}=e^{-x/2}x^{1/2-\mu }
\label{DkW_-mu+1/2,mu} \\
&&\left\{ \psi \left( 2\mu \right) +\frac{x}{2\mu }\,_{2}F_{2}\left( \left.
\begin{array}{c}
1,1 \\
2,1+2\mu%
\end{array}%
\right\vert -x\right) +e^{-2\pi i\mu }\,\Gamma \left( 1-2\mu ,x\right)
\,\gamma \left( 2\mu ,-x\right) \right\} .  \notag
\end{eqnarray}
\end{theorem}

\begin{proof}
Insert in (\ref{DkW_I1}) the reduction formula \cite[Eqn. 13.18.2]%
{olver2010nist} with $\kappa =-\mu +1/2$, i.e.%
\begin{equation}
\mathrm{W}_{1/2-\mu ,\mu }\left( x\right) =e^{-x/2}x^{1/2-\mu },
\label{W_1/2-mu,mu}
\end{equation}%
and the result given in (\ref{I*1_reduction}) to arrive at (\ref%
{DkW_-mu+1/2,mu}).
\end{proof}

\begin{remark}
If we consider (\ref{I*1_reduction_bis}), we obtain the following
alternative form:\
\begin{eqnarray}
&&\left. \frac{\partial \mathrm{W}_{\kappa ,\mu }\left( x\right) }{\partial
\kappa }\right\vert _{\kappa =-\mu +1/2}=e^{-x/2}x^{1/2-\mu }
\label{DkW_-mu+1/2,mu_bis} \\
&&\left\{ \psi \left( 2\mu \right) +\frac{\pi \left[ \cot \left( 2\pi \mu
\right) -i\right] }{\Gamma \left( 2\mu \right) }\gamma \left( 2\mu
,-x\right) +\frac{x}{2\mu -1}\,_{2}F_{2}\left( \left.
\begin{array}{c}
1,1 \\
2,2-2\mu%
\end{array}%
\right\vert x\right) \,\right\} .  \notag
\end{eqnarray}
\end{remark}

Table \ref{Table_3}\ shows the first derivative of $\mathrm{W}_{\kappa ,\mu
}\left( x\right) $ with respect to parameter $\kappa $ for some particular
values of $\kappa $ and $\mu $, and $x>0$, calculated with the aid of
MATHEMATICA program\ from (\ref{DkW_-mu+1/2,mu_bis}).

\begin{center}
\begin{table}[tbp] \centering%
\caption{Derivative of $W_{\kappa,\mu}$ with respect
to $\kappa$ by using (\ref{DkW_-mu+1/2,mu_bis}).}%
\rotatebox{90}{
\begin{tabular}{|c|c|c|}
\hline
$\kappa $ & $\mu $ & $\frac{\partial \mathrm{W}_{\kappa ,\mu }\left(
x\right) }{\partial \kappa }\quad \left( x>0\right) $ \\ \hline\hline
$-\frac{1}{4}$ & $\pm \frac{3}{4}$ & $x^{-1/4}e^{-x/2}\left[ 2-\gamma -\ln
4-2\,e^{x}\sqrt{\pi \,x}+\pi \,\mathrm{erfi}\left( \sqrt{x}\right)
+2x\,_{2}F_{2}\left( 1,1;\frac{1}{2},2;x\right) \right] $ \\ \hline
$\frac{1}{4}$ & $\pm \frac{1}{4}$ & $x^{1/4}e^{-x/2}\left[ \pi \,\mathrm{erfi%
}\left( \sqrt{x}\right) -2x\,_{2}F_{2}\left( 1,1;\frac{3}{2},2;x\right)
-\gamma -\ln 4\right] $ \\ \hline
$\frac{3}{4}$ & $\pm \frac{1}{4}$ & $e^{-x/2}\left\{ x^{3/4}\left[ 2-\gamma
-\ln 4+\pi \,\mathrm{erfi}\left( \sqrt{x}\right) -\frac{2}{3}%
x\,_{2}F_{2}\left( 1,1;\frac{5}{2},2;x\right) \right] -\sqrt{\pi }%
\,x^{1/4}\,e^{x}\right\} $ \\ \hline
$\frac{5}{4}$ & $\pm \frac{3}{4}$ & $\frac{1}{30}x^{-1/4}e^{-x/2}\left\{
2x^{3/2}\left[ 40-15\gamma -30\ln 2+15\pi \,\mathrm{erfi}\left( \sqrt{x}%
\right) -12x\,_{2}F_{2}\left( 1,1;\frac{7}{2},2;x\right) \right] -15\sqrt{%
\pi }e^{x}\left( 2x+1\right) \right\} $ \\ \hline
\end{tabular}%
}
\label{Table_3}%
\end{table}%
\end{center}

Notice that for $-2\mu =0,1,\ldots $, we obtain an indeterminate expression
in (\ref{DkW_-mu+1/2,mu}) and (\ref{DkW_-mu+1/2,mu_bis}). For these cases,
we present the following result.

\begin{theorem}
The following reduction formula holds true for $m=0,1,2,\ldots $:%
\begin{equation}
\left. \frac{\partial \mathrm{W}_{\kappa ,\mu }\left( x\right) }{\partial
\kappa }\right\vert _{\kappa =\left( 1+m\right) /2,\mu =\pm
m/2}=e^{-x/2}x^{\left( 1+m\right) /2}\left\{ \ln x-\sum_{k=1}^{m}x^{-k}\left[
e^{x}\,\Gamma \left( k\right) +\binom{m}{k}\,\gamma \left( k,-x\right) %
\right] \right\} .  \label{DkW_m}
\end{equation}
\end{theorem}

\begin{proof}
Take $\nu =2\mu $ in (\ref{DkW_-mu+1/2,mu_bis})\ and perform the limit $\nu
\rightarrow -m=0,-1,-2,\ldots $%
\begin{eqnarray}
&&\left. \frac{\partial \mathrm{W}_{\kappa ,\mu }\left( x\right) }{\partial
\kappa }\right\vert _{\kappa =\left( m+1\right) /2,\mu
=-m/2}=e^{-x/2}x^{\left( 1+m\right) /2}  \label{DkW_m_(1)} \\
&&\left\{ \lim_{\nu \rightarrow -m}\left[ \psi \left( \nu \right) +\frac{\pi %
\left[ \cot \left( \pi \nu \right) -i\right] }{\Gamma \left( \nu \right) }%
\gamma \left( \nu ,-x\right) \right] -\frac{x}{m+1}\,_{2}F_{2}\left( \left.
\begin{array}{c}
1,1 \\
2,2+m%
\end{array}%
\right\vert x\right) \right\} .  \notag
\end{eqnarray}%
On the one hand, let us prove the following asymptotic formulas for $\nu
\rightarrow -m=0,-1,-2,\ldots $
\begin{eqnarray}
\psi \left( \nu \right) &\approx &-\gamma +H_{m}-\frac{1}{\nu +m},
\label{Psi(nu)_asymp} \\
\pi \cot \left( \pi \nu \right) &\approx &\frac{1}{\nu +m},
\label{cot(x)_asymp} \\
\Gamma \left( \nu \right) &\approx &\frac{\left( -1\right) ^{m}}{m!}\frac{1}{%
\nu +m}.  \label{Gamma(nu)_asymp}
\end{eqnarray}%
In order to prove (\ref{Psi(nu)_asymp}), consider \cite[Eqn. 44:5:4]%
{oldham2009atlas}
\begin{eqnarray*}
\psi \left( \nu +m+1\right) &=&\psi \left( \nu \right) +\sum_{j=0}^{m}\frac{1%
}{\nu +j} \\
&=&\psi \left( \nu \right) +\sum_{j=1}^{m}\frac{1}{\nu +j-1}+\frac{1}{\nu +m}%
,
\end{eqnarray*}%
thus, knowing that \cite[Eqn. 1.3.6]{lebedev1965special}
\begin{equation}
\psi \left( 1\right) =-\gamma ,  \label{Psi(1)}
\end{equation}%
and performing the substitution $k=$ $j-m-1$, we have
\begin{eqnarray*}
\lim_{\nu \rightarrow -m}\psi \left( \nu \right) &=&\lim_{\nu \rightarrow
-m} \left[ -\gamma -\frac{1}{\nu +m}-\sum_{j=1}^{m}\frac{1}{j-m-1}\right] \\
&=&\lim_{\nu \rightarrow -m}\left[ -\gamma -\frac{1}{\nu +m}+H_{m}\right] ,
\end{eqnarray*}%
where $H_{n}=\sum_{k=1}^{n}\frac{1}{k}$ denotes the $n$-th harmonic number.
In order to prove (\ref{cot(x)_asymp}), note that $\cot x=\cot \left( x+\pi
\right) $ and for $x\in \left( -\pi ,\pi \right) $ we have the expansion
\cite[Eqn. 44:6:2]{oldham2009atlas}
\begin{equation*}
\cot x=\frac{1}{x}-\frac{x}{3}-\frac{x^{3}}{45}-\cdots
\end{equation*}%
Finally, notice that (\ref{Gamma(nu)_asymp})\ follows directly from \cite[%
Eqn. 1.1.5]{lebedev1965special}. Therefore, taking into account (\ref%
{Psi(nu)_asymp})-(\ref{Gamma(nu)_asymp}), and taking into account (\ref%
{Gamma=Lower+Upper}), we conclude
\begin{eqnarray}
&&\lim_{\nu \rightarrow -m}\left[ \psi \left( \nu \right) +\frac{\pi \left[
\cot \left( \pi \nu \right) -i\right] }{\Gamma \left( \nu \right) }\gamma
\left( \nu ,-x\right) \right]  \label{Limit_resultado} \\
&=&H_{m}-\gamma -i\pi +\left( -1\right) ^{m+1}m!\,\Gamma \left( -m,-x\right)
.  \notag
\end{eqnarray}%
Insert (\ref{Limit_resultado})\ in (\ref{DkW_m_(1)})\ to arrive at
\begin{eqnarray}
&&\left. \frac{\partial \mathrm{W}_{\kappa ,\mu }\left( x\right) }{\partial
\kappa }\right\vert _{\kappa =\left( m+1\right) /2,\mu
=-m/2}=e^{-x/2}x^{\left( 1+m\right) /2}  \label{DkW_m_a} \\
&&\left\{ H_{m}-\gamma -i\pi +\left( -1\right) ^{m+1}m!\,\Gamma \left(
-m,-x\right) -\frac{x}{m+1}\,_{2}F_{2}\left( \left.
\begin{array}{c}
1,1 \\
2,2+m%
\end{array}%
\right\vert x\right) \right\} .  \notag
\end{eqnarray}%
On the other hand, consider the reduction formula (\ref{2F2_reduction}),
derived in the Appendix,
\begin{equation}
_{2}F_{2}\left( \left.
\begin{array}{c}
1,1 \\
2,2+m%
\end{array}%
\right\vert x\right) =\frac{m+1}{x}\left\{ H_{m}-\mathrm{Ein}\left(
-x\right) +\sum_{k=1}^{m}\binom{m}{k}x^{-k}\gamma \left( k,-x\right)
\right\} ,  \label{2F2_reduction_m}
\end{equation}%
and the formula \cite[Eqn. 8.4.15]{olver2010nist}%
\begin{equation}
\Gamma \left( -m,z\right) =\frac{\left( -1\right) ^{m}}{m!}\left[ \mathrm{E}%
_{1}\left( z\right) -e^{-z}\sum_{k=0}^{m-1}\frac{\left( -1\right) ^{k}k!}{%
z^{k+1}}\right] ,  \label{Gamma(-m,z)}
\end{equation}%
where $\mathrm{E}_{1}\left( z\right) $ denotes the \textit{exponential
integral }\cite[Eqn. 6.2.1]{olver2010nist}, which is defined as
\begin{equation*}
\mathrm{E}_{1}\left( z\right) =\int_{z}^{\infty }\frac{e^{-t}}{t}dt,\quad
z\neq 0,
\end{equation*}%
where the path does not cross the negative real axis or pass throught the
origin. Also, consider the property \cite[Eqn. 6.2.4]{olver2010nist}%
\begin{equation}
\mathrm{E}_{1}\left( z\right) =\mathrm{Ein}\left( z\right) -\ln z-\gamma .
\label{E1(z)=Ein(z) -ln(z)-gamma}
\end{equation}%
Therefore, substituting (\ref{2F2_reduction_m})\ and (\ref{Gamma(-m,z)})\ in
(\ref{DkW_m_a}), and taking into account (\ref{E1(z)=Ein(z) -ln(z)-gamma}),
we arrive at (\ref{DkW_m}), as we wanted to prove.
\end{proof}

\begin{remark}
\label{Remark: Buschmann_1}It is worth noting that from \cite%
{buschman1974finite},
\begin{eqnarray}
&&\left. \frac{\partial \mathrm{W}_{\kappa ,\mu }\left( x\right) }{\partial
\kappa }\right\vert _{\kappa =\left( N+1\right) /2,\mu =M/2}=\mathrm{W}_{%
\frac{N+1}{2},\frac{N}{2}}\left( x\right) \ln x  \label{DkW_Buschmann} \\
&&+\sum_{k=1}^{\left( N+M\right) /2}\frac{\left( -1\right) ^{k}\left( \frac{%
N+M}{2}\right) !}{k\left( \frac{N+M}{2}-k\right) !}\mathrm{W}_{\frac{N+1}{2}%
-k,\frac{M}{2}}\left( x\right) +\sum_{k=1}^{\left( N-M\right) /2}\frac{%
\left( -1\right) ^{k}\left( \frac{N-M}{2}\right) !}{k\left( \frac{N-M}{2}%
-k\right) !}\mathrm{W}_{\frac{N+1}{2}-k,\frac{M}{2}}\left( x\right) ,  \notag
\end{eqnarray}%
where $-N\leq M\leq N$ and $M,N$ are integers of like parity, we can derive
an equivalent reduction formula to (\ref{DkW_m}). Indeed, taking $N=M=m$, (%
\ref{DkW_Buschmann})\ is reduced to%
\begin{equation}
\left. \frac{\partial \mathrm{W}_{\kappa ,\mu }\left( x\right) }{\partial
\kappa }\right\vert _{\kappa =\left( m+1\right) /2,\mu =m2}=\mathrm{W}_{%
\frac{m+1}{2},\frac{m}{2}}\left( x\right) \ln x+\sum_{k=1}^{m}\frac{\left(
-1\right) ^{k}\,m!}{k\left( m-k\right) !}\mathrm{W}_{\frac{m}{2}-k,\frac{m}{2%
}}\left( x\right) .  \label{Subs}
\end{equation}%
Note that from (\ref{W_k,k-1/2}), we have%
\begin{equation}
\mathrm{W}_{\frac{m+1}{2},\frac{m}{2}}\left( x\right) =e^{-x/2}x^{\left(
1+m\right) /2}.  \label{Subs_1}
\end{equation}%
Also, from (\ref{W_def_Tricomi})\ and the reduction formula for $n=0,1,...$
given in \cite[Eqn. 13.2.8]{olver2010nist}%
\begin{equation*}
\mathrm{U}\left( a,a+n+1,z\right) =z^{-a}\sum_{s=0}^{n}\binom{n}{s}\left(
a\right) _{s}z^{-s},
\end{equation*}%
we obtain%
\begin{equation}
\mathrm{W}_{\frac{m}{2}-k,\frac{m}{2}}\left( x\right) =\frac{%
e^{-x/2}x^{\left( 1+m\right) /2}x^{-k}}{\Gamma \left( k\right) }%
\sum_{s=0}^{m-k}\binom{m-k}{s}\Gamma \left( k+s\right) x^{-s}.
\label{Subs_2}
\end{equation}%
Therefore, susbtituting (\ref{Subs_1})\ and (\ref{Subs_2})\ in (\ref{Subs}),
and simplifying, we arrive at%
\begin{eqnarray}
&&\left. \frac{\partial \mathrm{W}_{\kappa ,\mu }\left( x\right) }{\partial
\kappa }\right\vert _{\kappa =\left( m+1\right) /2,\mu =m/2}  \notag \\
&=&e^{-x/2}x^{\left( 1+m\right) /2}\left[ \ln x+m!\sum_{k=1}^{m}\frac{\left(
-1\right) ^{k}}{k!}x^{-k}\sum_{s=0}^{m-k}\frac{\left( k+s-1\right) !}{%
s!\left( m-k-s\right) !}x^{-s}\right] .  \label{Sum_Buschmann}
\end{eqnarray}%
Perform the index substitution $s\rightarrow s+k$ and exchange the sum order
in (\ref{Sum_Buschmann}), to arrive at%
\begin{eqnarray}
&&\left. \frac{\partial \mathrm{W}_{\kappa ,\mu }\left( x\right) }{\partial
\kappa }\right\vert _{\kappa =\left( m+1\right) /2,\mu =m/2}  \notag \\
&=&e^{-x/2}x^{\left( 1+m\right) /2}\left[ \ln x+m!\sum_{s=1}^{m}\frac{x^{-s}%
}{s\left( m-s\right) !}\sum_{k=1}^{s}\binom{s}{k}\left( -1\right) ^{k}\right]
.  \label{Sum_Buschmann_2}
\end{eqnarray}%
By virtue of the binomial theorem, the inner sum in (\ref{Sum_Buschmann_2})\
is just $-1$, thus we finally obtain:%
\begin{equation}
\left. \frac{\partial \mathrm{W}_{\kappa ,\mu }\left( x\right) }{\partial
\kappa }\right\vert _{\kappa =\left( 1+m\right) /2,\mu =\pm
m/2}=e^{-x/2}x^{\left( 1+m\right) /2}\left[ \ln x-m!\sum_{k=1}^{m}\frac{%
x^{-k}}{k\left( m-k\right) !}\right] .  \label{DkW_m_Buschmann}
\end{equation}
\end{remark}

\begin{theorem}
For $n=0,1,2,\ldots $, and $x>0$, the following integral holds true:%
\begin{eqnarray}
&&I_{1}^{\ast }\left( \frac{n}{2},\frac{n+1}{2};x\right) =\frac{e^{x}\,%
\mathrm{Ein}\left( x\right) }{x^{n+1}}\Gamma \left( n+1,x\right)
+n!\sum_{k=0}^{n}\frac{x^{-k-1}}{\left( n-k\right) !}
\label{I*1_reduction_n} \\
&&\quad \left\{ \left( -1\right) ^{k+1}\Gamma \left( -k,x\right) \,\gamma
\left( k+1,-x\right) -H_{k}-\sum_{\ell =1}^{k}\binom{k}{\ell }\left(
-x\right) ^{-\ell }\gamma \left( \ell ,x\right) \right\} .  \notag
\end{eqnarray}
\end{theorem}

\begin{proof}
From (\ref{I*1_def}),\ we have
\begin{eqnarray*}
&&I_{1}^{\ast }\left( \mu -\frac{1}{2},\mu ;x\right) \\
&=&\int_{0}^{\infty }e^{-xt}\left( 1+t\right) ^{2\mu -1}\ln \left( \frac{1+t%
}{t}\right) dt,
\end{eqnarray*}%
thus, taking $\mu =\frac{n+1}{2}$ with $n=0,1,2,\ldots $ and applying the
binomial theorem, we get%
\begin{eqnarray}
I_{1}^{\ast }\left( \frac{n}{2},\frac{n+1}{2};x\right) &=&\sum_{k=0}^{n}%
\binom{n}{k}\int_{0}^{\infty }e^{-xt}t^{k}\ln \left( \frac{1+t}{t}\right) dt
\notag \\
&=&\sum_{k=0}^{n}\binom{n}{k}\mathcal{I}_{-}\left( k,x\right) .
\label{Suma_I(k,x)}
\end{eqnarray}%
Insert the result obtained in (\ref{I(nu,x)_resultado})\ for $\nu =k$ in (%
\ref{Suma_I(k,x)})\ to arrive at
\begin{eqnarray*}
&&I_{1}^{\ast }\left( \frac{n}{2},\frac{n+1}{2};x\right) =n!\sum_{k=0}^{n}%
\frac{x^{-k-1}}{\left( n-k\right) !} \\
&&\left\{ \left( -1\right) ^{k+1}\Gamma \left( -k,x\right) \,\gamma \left(
k+1,-x\right) +\frac{x}{k+1}\,_{2}F_{2}\left( \left.
\begin{array}{c}
1,1 \\
2,2+k%
\end{array}%
\right\vert -x\right) \right\} .
\end{eqnarray*}%
Now, take into account (\ref{2F2_reduction_m}), to get%
\begin{eqnarray}
&&I_{1}^{\ast }\left( \frac{n}{2},\frac{n+1}{2};x\right) =n!\sum_{k=0}^{n}%
\frac{x^{-k-1}}{\left( n-k\right) !}  \label{I*1_(1)} \\
&&\left\{ \left( -1\right) ^{k+1}\Gamma \left( -k,x\right) \,\gamma \left(
k+1,-x\right) -H_{k}+\mathrm{Ein}\left( x\right) -\sum_{\ell =1}^{k}\binom{k%
}{\ell }\left( -x\right) ^{-\ell }\gamma \left( \ell ,x\right) \right\} .
\notag
\end{eqnarray}%
Finally, note that using the \textit{exponential polynomial}, defined as%
\begin{equation*}
\mathrm{e}_{n}\left( x\right) =\sum_{k=0}^{n}\frac{x^{k}}{k!},
\end{equation*}%
and the property for $n=0,1,2,\ldots $ \cite[Eqn. 45:4:2]{oldham2009atlas}:%
\begin{equation*}
\Gamma \left( 1+n,x\right) =n!\,\mathrm{e}_{n}\left( x\right) \,e^{-x},
\end{equation*}%
we calculate the following finite sum as:
\begin{equation}
\sum_{k=0}^{n}\frac{x^{-k}}{\left( n-k\right) !}=x^{-n}\sum_{s=0}^{n}\frac{%
x^{s}}{s!}=\frac{x^{-n}e^{x}}{n!}\Gamma \left( 1+n,x\right) .
\label{Sum_Gamma}
\end{equation}%
Apply (\ref{Sum_Gamma}) to (\ref{I*1_(1)})\ in order to obtain (\ref%
{I*1_reduction_n}), as we wanted to prove.
\end{proof}

\begin{theorem}
For $n=0,1,2,\ldots $, the following reduction formula holds true:%
\begin{eqnarray}
&&\left. \frac{\partial \mathrm{W}_{\kappa ,\mu }\left( x\right) }{\partial
\kappa }\right\vert _{\kappa =n/2,\mu =\pm \left( n+1\right) /2}
\label{DkW_n/2,(n+1)/2} \\
&=&\,x^{-n/2}\,e^{x/2}\,\Gamma \left( 1+n,x\right) \left[ \mathrm{E}%
_{1}\left( x\right) +\ln x\right] +n!\,x^{n/2}e^{-x/2}  \notag \\
&&\sum_{k=0}^{n}\frac{x^{-k}}{\left( n-k\right) !}\left[ \left( -1\right)
^{k+1}\,\Gamma \left( -k,x\right) \,\gamma \left( k+1,-x\right)
-H_{k}-\sum_{\ell =1}^{k}\binom{k}{\ell }\left( -x\right) ^{-\ell }\gamma
\left( \ell ,x\right) \right] .  \notag
\end{eqnarray}
\end{theorem}

\begin{proof}
Applying (\ref{W_def_Tricomi}) and \cite[Eqn. 13.6.6]{olver2010nist}%
\begin{equation*}
\mathrm{U}\left( 1,2-a,z\right) =z^{a-1}e^{z}\,\Gamma \left( 1-a,z\right) ,
\end{equation*}%
see that for $n=0,1,2,\ldots $
\begin{equation}
\mathrm{W}_{n/2,\left( n+1\right) /2}\left( z\right) =z^{-n/2}e^{z/2}\Gamma
\left( 1+n,z\right) .  \label{W_n/2_reduction}
\end{equation}%
Taking into account (\ref{Psi(1)}) and (\ref{E1(z)=Ein(z) -ln(z)-gamma}),
insert (\ref{I*1_reduction_n})\ and (\ref{W_n/2_reduction})\ in (\ref{DkW_I1}%
)\ for $\kappa =\frac{n}{2}$ and $\mu =\frac{n+1}{2}$ to arrive at (\ref%
{DkW_n/2,(n+1)/2}), as we wanted to prove.
\end{proof}

\begin{theorem}
For $n=0,1,2,\ldots $, and $x>0$, the following integral holds true:%
\begin{eqnarray}
&&I_{1}^{\ast }\left( 0,n+\frac{1}{2};x\right) =\frac{n!e^{x/2}}{\sqrt{\pi }%
x^{n+1/2}}K_{n+1/2}\left( \frac{x}{2}\right) \,\mathrm{Ein}\left( x\right)
+\sum_{k=0}^{n}\binom{n}{k}\frac{\left( n+k\right) !}{x^{n+k+1}}
\label{I*1_reduction_0_n} \\
&&\left\{ \left( -1\right) ^{n+k+1}\Gamma \left( -n-k,x\right) \,\gamma
\left( n+k+1,-x\right) -H_{n+k}-\sum_{\ell =1}^{n+k}\binom{n+k}{\ell }\left(
-x\right) ^{-\ell }\gamma \left( \ell ,x\right) \right\} .  \notag
\end{eqnarray}
\end{theorem}

\begin{proof}
Applying the binomial theorem to (\ref{I*1_def}) for $\kappa =0$ and $\mu =n+%
\frac{1}{2}$,\ we have%
\begin{eqnarray}
I_{1}^{\ast }\left( 0,n+\frac{1}{2};x\right) &=&\sum_{k=0}^{n}\binom{n}{k}%
\int_{0}^{\infty }e^{-xt}t^{n+k}\ln \left( \frac{1+t}{t}\right) dt  \notag \\
&=&\sum_{k=0}^{n}\binom{n}{k}\mathcal{I}_{-}\left( n+k,x\right)
\label{Suma_I(n+k,x)}
\end{eqnarray}%
Insert\ the result obtained in (\ref{I(nu,x)_resultado})\ for $\nu =n+k$ in (%
\ref{Suma_I(n+k,x)}) to get
\begin{eqnarray*}
&&I_{1}^{\ast }\left( 0,n+\frac{1}{2};x\right) =\sum_{k=0}^{n}\binom{n}{k}%
\frac{\left( n+k\right) !}{x^{n+k+1}} \\
&&\left\{ \left( -1\right) ^{n+k+1}\Gamma \left( -n-k,x\right) \,\gamma
\left( n+k+1,-x\right) +\frac{x}{n+k+1}\,_{2}F_{2}\left( \left.
\begin{array}{c}
1,1 \\
2,2+n+k%
\end{array}%
\right\vert -x\right) \right\} .
\end{eqnarray*}%
Now, take into account (\ref{2F2_reduction_m}), to obtain%
\begin{eqnarray*}
&&I_{1}^{\ast }\left( 0,n+\frac{1}{2};x\right) =\sum_{k=0}^{n}\binom{n}{k}%
\frac{\left( n+k\right) !}{x^{n+k+1}} \\
&&\left\{ \left( -1\right) ^{n+k+1}\Gamma \left( -n-k,x\right) \,\gamma
\left( n+k+1,-x\right) +\,\mathrm{Ein}\left( x\right) -H_{n+k}-\sum_{\ell
=1}^{n+k}\binom{n+k}{\ell }\left( -x\right) ^{-\ell }\gamma \left( \ell
,x\right) \right\} .
\end{eqnarray*}%
Finally, consider \cite[Eqns. 10.47.9,12]{olver2010nist}%
\begin{equation}
\sqrt{\frac{z}{\pi }}K_{n+1/2}\left( \frac{z}{2}\right) =\frac{z}{\pi }\,%
\mathrm{k}_{n}\left( \frac{z}{2}\right) =e^{-z/2}\sum_{k=0}^{n}\frac{\left(
n+k\right) !\,z^{-k}}{k!\left( n-k\right) !},  \label{Sum_BesselK_n+1/2}
\end{equation}%
where $\mathrm{k}_{n}\left( z\right) $ is the \textit{modified spherical
Bessel function of the second kind}, to arrive at the desired result.
\end{proof}

\begin{theorem}
For $n=0,1,2,\ldots $, the following reduction formula holds true:%
\begin{eqnarray}
&&\left. \frac{\partial \mathrm{W}_{\kappa ,\mu }\left( x\right) }{\partial
\kappa }\right\vert _{\kappa =0,\mu =\pm \left( n+1/2\right) }  \label{DkW0n}
\\
&=&\sqrt{\frac{x}{\pi }}K_{n+1/2}\left( \frac{x}{2}\right) \left[ H_{n}+%
\mathrm{E}_{1}\left( x\right) +\ln x\right] +e^{-x/2}\sum_{k=0}^{n}\frac{%
\left( n+k\right) !\,x^{-k}}{k!\left( n-k\right) !}  \notag \\
&&\quad \left[ \left( -1\right) ^{n+k+1}\,\Gamma \left( -n-k,x\right)
\,\gamma \left( n+k+1,-x\right) -H_{n+k}-\sum_{\ell =1}^{n+k}\binom{n+k}{%
\ell }\left( -x\right) ^{-\ell }\gamma \left( \ell ,x\right) \right] .
\notag
\end{eqnarray}
\end{theorem}

\begin{proof}
Take $\kappa =0$ and $\mu =n+\frac{1}{2}$ in (\ref{DkW_I1}), to obtain%
\begin{equation}
\left. \frac{\partial \mathrm{W}_{\kappa ,\mu }\left( x\right) }{\partial
\kappa }\right\vert _{\kappa =0,\mu =n+1/2}=\psi \left( n+1\right) \mathrm{W}%
_{0,n+1/2}\left( x\right) +\frac{x^{n+1}e^{-x/2}}{n!}\,I_{1}^{\ast }\left(
0,n+\frac{1}{2};x\right) .  \label{DkWn_1}
\end{equation}%
Consider \cite[Eqn. 1.3.7]{lebedev1965special}%
\begin{equation}
\psi \left( n+1\right) =-\gamma +H_{n},  \label{Psi(n+1)}
\end{equation}%
and \cite[Eqns. 13.18.9]{olver2010nist}%
\begin{equation}
\mathrm{W}_{0,n+1/2}\left( z\right) =\sqrt{\frac{z}{\pi }}K_{n+1/2}\left(
\frac{z}{2}\right) .  \label{W_0,n+1/2}
\end{equation}%
Substitute (\ref{I*1_reduction_0_n}), (\ref{Psi(n+1)})\ and (\ref{W_0,n+1/2}%
)\ in (\ref{DkWn_1}),\ and take into account (\ref{DkW_mu=DkW_-mu}) and (\ref%
{E1(z)=Ein(z) -ln(z)-gamma}),\ to arrive at (\ref{DkW0n}), as we wanted to
prove.
\end{proof}

Table \ref{Table_3A} shows the first derivative of $\mathrm{W}_{\kappa ,\mu
}\left( x\right) $ with respect to parameter $\kappa $ for some particular
values of $\kappa $ and $\mu $, calculated with the aid of MATHEMATICA\ from
(\ref{DkW_m}), (\ref{DkW_n/2,(n+1)/2}) and (\ref{DkW0n}).

\begin{center}
\begin{table}[tbp] \centering%
\caption{First derivative of $\mathrm{W}_{\kappa,\mu}(x)$ with respect to parameter $\kappa$ for particular values of $\kappa$ and
$\mu$.}%
\begin{tabular}{|c|c|c|}
\hline
$\kappa $ & $\mu $ & $\frac{\partial \mathrm{W}_{\kappa ,\mu }\left(
x\right) }{\partial \kappa }$ \\ \hline\hline
$0$ & $\pm \frac{1}{2}$ & $e^{-x/2}\left[ \ln x+e^{x}\,\Gamma \left(
0,x\right) \right] $ \\ \hline
$0$ & $\pm \frac{3}{2}$ & $x^{-1}e^{-x/2}\left\{ \left( x-2\right) e^{x}%
\left[ \mathrm{Chi}\left( x\right) -\mathrm{Shi}\left( x\right) \right]
+\left( x+2\right) \ln x+2\right\} $ \\ \hline
$0$ & $\pm \frac{5}{2}$ & $x^{-2}e^{-x/2}\left\{ \left( x^{2}+6x+12\right)
\ln x+18-\left( x^{2}-6x+12\right) e^{x}\left[ \mathrm{Chi}\left( x\right) -%
\mathrm{Shi}\left( x\right) \right] \right\} $ \\ \hline
$\frac{1}{2}$ & $0$ & $\sqrt{x}e^{-x/2}\ln x$ \\ \hline
$\frac{1}{2}$ & $\pm 1$ & $x^{-1/2}e^{-x/2}\left[ \left( x+1\right) \ln
x+e^{x}\,\Gamma \left( 0,x\right) \right] $ \\ \hline
$1$ & $\pm \frac{1}{2}$ & $e^{-x/2}\left( x\ln x-1\right) $ \\ \hline
$1$ & $\pm \frac{3}{2}$ & $x^{-1}e^{-x/2}\left\{ \left( x^{2}+2x+2\right)
\ln x-2\,e^{x}\left[ \mathrm{Chi}\left( x\right) -\mathrm{Shi}\left(
x\right) \right] -x\right\} $ \\ \hline
$\frac{3}{2}$ & $\pm 1$ & $x^{-1/2}e^{-x/2}\left( x^{2}\ln x-2x-1\right) $
\\ \hline
$\frac{3}{2}$ & $\pm 2$ & $x^{-3/2}e^{-x/2}\left\{ \left(
x^{3}+3x^{2}+6x+6\right) \ln x-2x^{2}-4-6\,e^{x}\left[ \mathrm{Chi}\left(
x\right) -\mathrm{Shi}\left( x\right) \right] \right\} $ \\ \hline
$2$ & $\pm \frac{3}{2}$ & $e^{-x/2}\left( x^{2}\ln x-3x-3-\frac{2}{x}\right)
$ \\ \hline
\end{tabular}%
\label{Table_3A}%
\end{table}%
\end{center}

\subsection{Application to the calculation of infinite integrals}

Additional integral representations of the Whittaker function $\mathrm{W}%
_{\kappa ,\mu }\left( x\right) $ in terms of Bessel functions \cite[Sect.
7.4.2]{magnus2013formulas} are known:%
\begin{eqnarray}
&&\mathrm{W}_{\kappa ,\mu }\left( x\right)  \notag \\
&=&\frac{\,2\sqrt{x}e^{-x/2}}{\Gamma \left( \frac{1}{2}+\mu -\kappa \right)
\Gamma \left( \frac{1}{2}-\mu -\kappa \right) }\int_{0}^{\infty
}e^{-t}t^{-\kappa -1/2}K_{2\mu }\left( 2\sqrt{xt}\right) dt
\label{Int_W_BesselK} \\
&&\mathrm{Re}\left( \frac{1}{2}\pm \mu -\kappa \right) >0.  \notag
\end{eqnarray}

Let us introduce the following infinite logarithmic integral.

\begin{definition}
\begin{equation}
\mathcal{H}\left( \kappa ,\mu ;x\right) =\int_{0}^{\infty }e^{-t}t^{-\kappa
-1/2}K_{2\mu }\left( 2\sqrt{xt}\right) \ln t\,dt.  \label{H_def}
\end{equation}
\end{definition}

\begin{theorem}
For $\kappa ,\mu \in
\mathbb{R}
$ with $\left\vert \mu \right\vert <\frac{1}{2}-\kappa $, the following
integral holds true:%
\begin{eqnarray}
&&\mathcal{H}\left( \kappa ,\mu ;x\right) =\frac{1}{2}\,\Gamma \left( \frac{1%
}{2}-\mu -\kappa \right)  \label{H->I*1} \\
&&\left\{ \frac{\Gamma \left( \frac{1}{2}+\mu -\kappa \right) \psi \left(
\frac{1}{2}-\mu -\kappa \right) }{\sqrt{x}\,e^{-x/2}}\mathrm{W}_{\kappa ,\mu
}\left( x\right) +x^{\mu }\,I_{1}^{\ast }\left( \kappa ,\mu ;x\right)
\right\} ,  \notag
\end{eqnarray}%
where $I_{1}^{\ast }\left( \kappa ,\mu ;x\right) $ is given by (\ref%
{I1s_general}).
\end{theorem}

\begin{proof}
Differentiation of (\ref{Int_W_BesselK})\ with respect to parameter $\kappa $
yields:%
\begin{eqnarray}
&&\frac{\partial \mathrm{W}_{\kappa ,\mu }\left( x\right) }{\partial \kappa }%
=\left[ \psi \left( \frac{1}{2}-\mu -\kappa \right) +\psi \left( \frac{1}{2}%
+\mu -\kappa \right) \right] \mathrm{W}_{\kappa ,\mu }\left( x\right)
\label{DkW_H_def} \\
&&-\frac{2\sqrt{x}e^{-x/2}}{\Gamma \left( \frac{1}{2}+\mu -\kappa \right)
\Gamma \left( \frac{1}{2}-\mu -\kappa \right) }\mathcal{H}\left( \kappa ,\mu
;x\right)  \notag
\end{eqnarray}%
Equate (\ref{DkW_I1})\ to (\ref{DkW_H_def}) to arrive at (\ref{H->I*1}), as
we wanted to prove.
\end{proof}

\subsection{Derivative with respect to the second parameter $\partial
\mathrm{W}_{\protect\kappa ,\protect\mu }\left( x\right) /\partial \protect%
\mu $}

First, note that
\begin{equation}
\frac{\partial \mathrm{W}_{\kappa ,\pm \mu }\left( x\right) }{\partial \mu }%
=\pm \frac{\partial \mathrm{W}_{\kappa ,\mu }\left( x\right) }{\partial \mu }%
,  \label{DmW_pm}
\end{equation}%
since (\ref{W_mu=W_-mu})\ is satisfied. Next, let us introduce the following
definitions in order to calculate the first derivative of $\mathrm{W}%
_{\kappa ,\mu }\left( x\right) $ with respect to parameter $\mu $.

\begin{definition}
Following the notation introduced in (\ref{G(1)_def})-(\ref{H(1)_def}),
define%
\begin{equation}
\tilde{G}^{\left( 1\right) }\left( a,b,z\right) =\frac{\partial }{\partial a}%
\,\left[ \mathrm{U}\left( a,b,z\right) \right] ,  \label{G1_tilde_def}
\end{equation}%
and%
\begin{equation}
\tilde{H}^{\left( 1\right) }\left( a,b,z\right) =\frac{\partial }{\partial b}%
\,\left[ \mathrm{U}\left( a,b,z\right) \right] .  \label{H1_tilde_def}
\end{equation}
\end{definition}

Direct differentiation of (\ref{W_def_Tricomi}) yields:%
\begin{eqnarray}
&&\frac{\partial \mathrm{W}_{\kappa ,\mu }\left( x\right) }{\partial \mu }
\label{DmW_Tricomi} \\
&=&\ln x\,\mathrm{W}_{\kappa ,\mu }\left( x\right) +x^{\mu +1/2}e^{-x/2}
\notag \\
&&\left[ \tilde{G}^{\left( 1\right) }\left( \frac{1}{2}-\kappa +\mu ,1+2\mu
,x\right) +2\,\tilde{H}^{\left( 1\right) }\left( \frac{1}{2}-\kappa +\mu
,1+2\mu ,x\right) \right]  \notag
\end{eqnarray}

\begin{definition}
For $\mathrm{Re}\left( \mu -\kappa \right) >-\frac{1}{2}$ and $x>0$, define:%
\begin{eqnarray}
I_{3}^{\ast }\left( \kappa ,\mu ;x\right) &=&\int_{0}^{\infty }e^{-xt}t^{\mu
-\kappa -1/2}\left( 1+t\right) ^{\mu +\kappa -1/2}\ln \left[ t\left(
1+t\right) \right] dt,  \label{I*3_def} \\
I_{4}^{\ast }\left( \kappa ,\mu ;x\right) &=&\int_{1}^{\infty }e^{-xt}t^{\mu
+\kappa -1/2}\left( t-1\right) ^{\mu -\kappa -1/2}\ln \left[ t\left(
t-1\right) \right] dt.  \label{I*4_def}
\end{eqnarray}
\end{definition}

These integrals are interrelated by%
\begin{equation*}
I_{4}^{\ast }\left( \kappa ,\mu ;x\right) =e^{-x}I_{3}^{\ast }\left( \kappa
,\mu ;x\right) .
\end{equation*}

Differentiation of (\ref{W_k,mu_int_1})\ with respect to parameter $\mu $
gives%
\begin{eqnarray}
&&\frac{\partial \mathrm{W}_{\kappa ,\mu }\left( x\right) }{\partial \mu }
\label{DmW_I*3} \\
&=&\left[ \ln x-\psi \left( \mu -\kappa +\frac{1}{2}\right) \right] \mathrm{W%
}_{\kappa ,\mu }\left( x\right) +\frac{x^{\mu +1/2}e^{-x/2}}{\Gamma \left(
\mu -\kappa +\frac{1}{2}\right) }I_{3}^{\ast }\left( \kappa ,\mu ;x\right) .
\notag
\end{eqnarray}

\begin{theorem}
According to the notation introduced in (\ref{G1_tilde_def})\ and (\ref%
{H1_tilde_def}), the following integral holds true for $x>0$:
\begin{eqnarray}
&&I_{3}^{\ast }\left( \kappa ,\mu ;x\right)  \label{I*3_general} \\
&=&\Gamma \left( \frac{1}{2}-\kappa +\mu \right) \left\{ \mathrm{U}\left(
\frac{1}{2}-\kappa +\mu ,1+2\mu ,x\right) \psi \left( \frac{1}{2}-\kappa
+\mu \right) \right.  \notag \\
&&\,+\left. \tilde{G}^{\left( 1\right) }\left( \frac{1}{2}-\kappa +\mu
,1+2\mu ,x\right) +2\,\tilde{H}^{\left( 1\right) }\left( \frac{1}{2}-\kappa
+\mu ,1+2\mu ,x\right) \right\} .  \notag
\end{eqnarray}
\end{theorem}

\begin{proof}
Comparing (\ref{DmW_Tricomi})\ to (\ref{DmW_I*3}), taking into account (\ref%
{W_def_Tricomi}), we arrive at (\ref{I*3_general}), as we wanted to prove.
\end{proof}

\begin{theorem}
For $-2\mu \neq 0,1,2,\ldots $ and $x>0$, the following reduction formula
holds true:%
\begin{eqnarray}
&&\left. \frac{\partial \mathrm{W}_{\kappa ,\mu }\left( x\right) }{\partial
\mu }\right\vert _{\kappa =1/2-\mu }=x^{1/2-\mu }e^{-x/2}  \label{DmW_1/2-mu}
\\
&&\left\{ \frac{x}{2\mu }\,_{2}F_{2}\left( \left.
\begin{array}{c}
1,1 \\
2,1+2\mu%
\end{array}%
\right\vert -x\right) +e^{-2\pi i\mu }\,\Gamma \left( 1-2\mu ,x\right)
\,\gamma \left( 2\mu ,-x\right) +\psi \left( 2\mu \right) -\ln x\right\} .
\notag
\end{eqnarray}
\end{theorem}

\begin{proof}
According to (\ref{I(nu,x)_resultado}) and (\ref{I*3_def}), note that%
\begin{eqnarray}
&&I_{3}^{\ast }\left( \frac{1}{2}-\mu ,\mu ;x\right) =\mathcal{I}_{+}\left(
2\mu -1,x\right)  \label{I*3(1/2-mu,mu,x)} \\
&=&\frac{\Gamma \left( 2\mu \right) }{x^{2\mu }}\left\{ \frac{x}{2\mu }%
\,_{2}F_{2}\left( \left.
\begin{array}{c}
1,1 \\
2,1+2\mu%
\end{array}%
\right\vert -x\right) \right.  \notag \\
&&+\left. e^{-2\pi i\mu }\,\Gamma \left( 1-2\mu ,x\right) \,\gamma \left(
2\mu ,-x\right) +2\left[ \psi \left( 2\mu \right) -\ln x\right] \right\} .
\notag
\end{eqnarray}%
Taking $\kappa =1/2-\mu $ in (\ref{DmW_I*3}), substitute (\ref%
{I*3(1/2-mu,mu,x)})\ and (\ref{W_1/2-mu,mu})\ to arrive at the desired
result given in (\ref{DmW_1/2-mu}).
\end{proof}

\begin{remark}
If we take into account (\ref{I(nu,x)_2})\ in (\ref{I*3(1/2-mu,mu,x)}), we
obtain the alternative form:%
\begin{eqnarray*}
&&I_{3}^{\ast }\left( \frac{1}{2}-\mu ,\mu ;x\right) \\
&=&\frac{1}{x^{2\mu }}\left\{ \pi \left[ \cot \left( 2\pi \mu \right) -i%
\right] \gamma \left( 2\mu ,-x\right) +x\,\Gamma \left( 2\mu -1\right)
\,_{2}F_{2}\left( \left.
\begin{array}{c}
1,1 \\
2,2-2\mu%
\end{array}%
\right\vert x\right) \right. \\
&&+\left.
\begin{array}{c}
\\
\end{array}%
2\,\Gamma \left( 2\mu \right) \left[ \psi \left( 2\mu \right) -\ln x\right]
\right\} ,
\end{eqnarray*}%
thus for $-2\mu \neq 0,1,2,\ldots $ and $x>0$, we have
\begin{eqnarray}
&&\left. \frac{\partial \mathrm{W}_{\kappa ,\mu }\left( x\right) }{\partial
\mu }\right\vert _{\kappa =1/2-\mu }=x^{1/2-\mu }e^{-x/2}
\label{DmW_1/2-mu_alternative} \\
&&\left\{ \frac{\pi \left[ \cot \left( 2\pi \mu \right) -i\right] }{\Gamma
\left( 2\mu \right) }\gamma \left( 2\mu ,-x\right) +\frac{x}{2\mu -1}%
\,_{2}F_{2}\left( \left.
\begin{array}{c}
1,1 \\
2,2-2\mu%
\end{array}%
\right\vert x\right) +\psi \left( 2\mu \right) -\ln x\right\} .  \notag
\end{eqnarray}
\end{remark}

Table \ref{Table_4}\ shows the first derivative of $\mathrm{W}_{\kappa ,\mu
}\left( x\right) $ with respect to parameter $\mu $ for some particular
values of $\kappa $ and $\mu $, with $x>0$, calculated from (\ref%
{DmW_1/2-mu_alternative}) with the aid of MATHEMATICA\ program.

\begin{center}
\begin{table}[htbp] \centering%
\caption{Derivative of $W_{\kappa,\mu}$ with respect
to $\mu$ by using (\ref{DmW_1/2-mu_alternative}).}%
\rotatebox{90}{
\begin{tabular}{|c|c|c|}
\hline
$\kappa $ & $\mu $ & $\frac{\partial \mathrm{W}_{\kappa ,\mu }\left(
x\right) }{\partial \mu }\quad \left( x>0\right) $ \\ \hline\hline
$-\frac{1}{4}$ & $\pm \frac{3}{4}$ & $\pm x^{-1/4}e^{-x/2}\left[ 2-\gamma
-\ln \left( 4x\right) -2\,e^{x}\sqrt{\pi \,x}+\pi \,\mathrm{erfi}\left(
\sqrt{x}\right) +2x\,_{2}F_{2}\left( 1,1;\frac{1}{2},2;x\right) \right] $ \\
\hline
$\frac{1}{4}$ & $\pm \frac{1}{4}$ & $\pm x^{1/4}e^{-x/2}\left[ \pi \,\mathrm{%
erfi}\left( \sqrt{x}\right) -2x\,_{2}F_{2}\left( 1,1;\frac{3}{2},2;x\right)
-\gamma -\ln \left( 4x\right) \right] $ \\ \hline
$\frac{3}{4}$ & $\pm \frac{1}{4}$ & $\pm e^{-x/2}\left\{ x^{3/4}\left[ \frac{%
2}{3}x\,_{2}F_{2}\left( 1,1;\frac{5}{2},2;x\right) -2+\gamma +\ln \left(
4x\right) -\pi \,\mathrm{erfi}\left( \sqrt{x}\right) \right] +\sqrt{\pi }%
\,x^{1/4}\,e^{x}\right\} $ \\ \hline
$\frac{5}{4}$ & $\pm \frac{3}{4}$ & $%
\begin{array}{l}
\pm \frac{1}{30}x^{-1/4}e^{-x/2}\left\{ 15\sqrt{\pi }e^{x}\left( 2x+1\right)
-2\,x^{3/2}\right. \\
\left. \left[ 40-15\gamma -30\ln \left( 2x\right) +15\pi \,\mathrm{erfi}%
\left( \sqrt{x}\right) -12x\,_{2}F_{2}\left( 1,1;\frac{7}{2},2;x\right) %
\right] \right\}%
\end{array}%
$ \\ \hline
\end{tabular}%
}
\label{Table_4}%
\end{table}%
\end{center}

Notice that for $-2\mu =0,1,\ldots $, we obtain an indeterminate expression
in (\ref{DmW_1/2-mu}) or (\ref{DmW_1/2-mu_alternative}). For these cases, we
present the following result.

\begin{theorem}
The following reduction formula holds true for $m=0,1,2,\ldots $:%
\begin{eqnarray}
&&\left. \frac{\partial \mathrm{W}_{\kappa ,\mu }\left( x\right) }{\partial
\mu }\right\vert _{\kappa =\left( 1+m\right) /2,\mu =\pm m/2}
\label{DmW_m_limt} \\
&=&\pm e^{-x/2}x^{\left( 1+m\right) /2}\sum_{k=1}^{m}x^{-k}\left[
e^{x}\,\Gamma \left( k\right) +\binom{m}{k}\,\gamma \left( k,-x\right) %
\right] .  \notag
\end{eqnarray}
\end{theorem}

\begin{proof}
Take $\nu =2\mu $ in (\ref{DmW_1/2-mu_alternative})\ and perform the limit $%
\nu \rightarrow -m=0,-1,-2,\ldots $%
\begin{eqnarray*}
&&\left. \frac{\partial \mathrm{W}_{\kappa ,\mu }\left( x\right) }{\partial
\mu }\right\vert _{\kappa =\left( m+1\right) /2,\mu =-m/2}=e^{-x/2}x^{\left(
1+m\right) /2} \\
&&\left\{ \lim_{\nu \rightarrow -m}\left[ \psi \left( \nu \right) +\frac{\pi %
\left[ \cot \left( \pi \nu \right) -i\right] }{\Gamma \left( \nu \right) }%
\gamma \left( \nu ,-x\right) \right] -\frac{x}{m+1}\,_{2}F_{2}\left( \left.
\begin{array}{c}
1,1 \\
2,2+m%
\end{array}%
\right\vert x\right) -\ln x\right\} .
\end{eqnarray*}%
Applying the result given in (\ref{Limit_resultado}), we get%
\begin{eqnarray}
&&\left. \frac{\partial \mathrm{W}_{\kappa ,\mu }\left( x\right) }{\partial
\mu }\right\vert _{\kappa =\left( m+1\right) /2,\mu =-m/2}=e^{-x/2}x^{\left(
1+m\right) /2}  \label{DmW_m_a} \\
&&\left\{ H_{m}-\gamma -i\,\pi +\left( -1\right) ^{m+1}m!\,\Gamma \left(
-m,-x\right) -\frac{x}{m+1}\,_{2}F_{2}\left( \left.
\begin{array}{c}
1,1 \\
2,2+m%
\end{array}%
\right\vert x\right) -\ln x\right\} .  \notag
\end{eqnarray}%
Now, compare (\ref{DkW_m})\ to (\ref{DkW_m_a}), to see that
\begin{eqnarray}
&&H_{m}-\gamma -i\,\pi +\left( -1\right) ^{m+1}m!\,\Gamma \left(
-m,-x\right) -\frac{x}{m+1}\,_{2}F_{2}\left( \left.
\begin{array}{c}
1,1 \\
2,2+m%
\end{array}%
\right\vert x\right)  \notag \\
&=&\ln x-\sum_{k=1}^{m}x^{-k}\left[ e^{x}\,\Gamma \left( k\right) +\binom{m}{%
k}\,\gamma \left( k,-x\right) \right]  \label{Comparison}
\end{eqnarray}%
Therefore, inserting (\ref{Comparison})\ in (\ref{DmW_m_a}), and taking into
account (\ref{DmW_pm}), we arrive at (\ref{DmW_m_limt}), as we wanted to
prove.
\end{proof}

\begin{remark}
It is worth noting that from \cite{buschman1974finite},
\begin{eqnarray}
&&\left. \frac{\partial \mathrm{W}_{\kappa ,\mu }\left( x\right) }{\partial
\mu }\right\vert _{\kappa =\left( N+1\right) /2,\mu =M/2}
\label{DmW_Buschmann} \\
= &&\sum_{k=1}^{\left( N+M\right) /2}\frac{\left( -1\right) ^{k}\left( \frac{%
N+M}{2}\right) !}{k\left( \frac{N+M}{2}-k\right) !}\mathrm{W}_{\frac{N+1}{2}%
-k,\frac{M}{2}}\left( x\right) +\sum_{k=1}^{\left( N-M\right) /2}\frac{%
\left( -1\right) ^{k}\left( \frac{N-M}{2}\right) !}{k\left( \frac{N-M}{2}%
-k\right) !}\mathrm{W}_{\frac{N+1}{2}-k,\frac{M}{2}}\left( x\right) ,  \notag
\end{eqnarray}%
where $-N\leq M\leq N$ and $M,N$ are integers of like parity, we can derive
an equivalent reduction formula to (\ref{DmW_m_limt}). Indeed, following
similar steps as in Remark \ref{Remark: Buschmann_1}, we arrive at:%
\begin{equation}
\left. \frac{\partial \mathrm{W}_{\kappa ,\mu }\left( x\right) }{\partial
\mu }\right\vert _{\kappa =\left( 1+m\right) /2,\mu =\pm m/2}=\pm
m!\,e^{-x/2}x^{\left( 1+m\right) /2}\sum_{k=1}^{m}\frac{x^{-k}}{k\left(
m-k\right) !}.  \label{DmW_m_Buschmann}
\end{equation}
\end{remark}

\begin{theorem}
For $n=0,1,2,\ldots $, the following reduction formula holds true:%
\begin{eqnarray}
&&\left. \frac{\partial \mathrm{W}_{\kappa ,\mu }\left( x\right) }{\partial
\mu }\right\vert _{\kappa =n/2,\mu =\pm \left( n+1\right) /2}
\label{DmW_n/2} \\
&=&\pm x^{-n/2}\,e^{x/2}\,\,\mathrm{E}_{1}\left( x\right) \,\,\Gamma \left(
1+n,x\right) \pm n!\,x^{n/2}e^{-x/2}  \notag \\
&&\sum_{k=0}^{n}\frac{x^{-k}}{\left( n-k\right) !}\left\{ H_{k}+\left(
-1\right) ^{k+1}\,\Gamma \left( -k,x\right) \,\gamma \left( k+1,-x\right)
-\sum_{\ell =1}^{k}\binom{k}{\ell }\left( -x\right) ^{-\ell }\gamma \left(
\ell ,x\right) \right\} .  \notag
\end{eqnarray}
\end{theorem}

\begin{proof}
According to (\ref{I*3_def}) and (\ref{I(nu,x)_resultado}), using the
binomial theorem, and taking into account (\ref{Psi(n+1)}), we have
\begin{eqnarray}
&&I_{3}^{\ast }\left( \frac{n}{2},\frac{n+1}{2};x\right) =\int_{0}^{\infty
}e^{-xt}\left( 1+t\right) ^{n}\ln \left[ t\left( 1+t\right) \right] dt
\label{I*3_n/2} \\
&=&\sum_{k=0}^{n}\binom{n}{k}\mathcal{I}_{+}\left( k,x\right)
=n!\sum_{k=0}^{n}\frac{x^{-k-1}}{\left( n-k\right) !}  \notag \\
&&\left\{ \frac{x}{k+1}\,_{2}F_{2}\left( \left.
\begin{array}{c}
1,1 \\
2,2+k%
\end{array}%
\right\vert -x\right) +\left( -1\right) ^{k+1}\,\Gamma \left( -k,x\right)
\,\gamma \left( k+1,-x\right) +2\left[ H_{k}-\gamma -\ln x\right] \right\} .
\notag
\end{eqnarray}%
Consider (\ref{2F2_reduction_m}), (\ref{E1(z)=Ein(z) -ln(z)-gamma})\ and (%
\ref{Sum_Gamma})\ in order to rewrite (\ref{I*3_n/2})\ as
\begin{eqnarray}
&&I_{3}^{\ast }\left( \frac{n}{2},\frac{n+1}{2};x\right) =\frac{\mathrm{E}%
_{1}\left( x\right) -\ln x-\gamma }{x^{n+1}}\,e^{x}\,\Gamma \left(
n+1,x\right)  \label{I*3_n/2_(1)} \\
&&+n!\sum_{k=0}^{n}\frac{x^{-k-1}}{\left( n-k\right) !}\left\{ H_{k}+\left(
-1\right) ^{k+1}\,\Gamma \left( -k,x\right) \,\gamma \left( k+1,-x\right)
-\sum_{\ell =1}^{k}\binom{k}{\ell }\left( -x\right) ^{-\ell }\gamma \left(
\ell ,x\right) \right\} .  \notag
\end{eqnarray}%
Therefore, substituting (\ref{W_n/2_reduction}), (\ref{Psi(1)}), and (\ref%
{I*3_n/2_(1)})\ in (\ref{DmW_I*3}), we obtain (\ref{DmW_n/2}), as we wanted
to prove.
\end{proof}

\begin{theorem}
For $n=0,1,2,\ldots $, the following reduction formula holds true:%
\begin{eqnarray}
&&\left. \frac{\partial \mathrm{W}_{\kappa ,\mu }\left( x\right) }{\partial
\mu }\right\vert _{\kappa =0,\mu =\pm \left( n+1/2\right) }
\label{DmW_0_n+1/2} \\
&=&\pm \sqrt{\frac{x}{\pi }}K_{n+1/2}\left( \frac{x}{2}\right) \left[
\mathrm{E}_{1}\left( x\right) -H_{n}\right] \pm e^{-x/2}\sum_{k=0}^{n}\frac{%
\left( n+k\right) !\,x^{-k}}{k!\left( n-k\right) !}  \notag \\
&&\left\{ H_{n+k}+\left( -1\right) ^{n+k+1}\,\Gamma \left( -n-k,x\right)
\,\gamma \left( n+k+1,-x\right) -\sum_{\ell =1}^{k}\binom{k}{\ell }\left(
-x\right) ^{-\ell }\gamma \left( \ell ,x\right) \right\} .  \notag
\end{eqnarray}
\end{theorem}

\begin{proof}
Applying the binomial theorem to (\ref{I*3_def}) for $\kappa =0$ and $\mu =n+%
\frac{1}{2}$,\ and taking into account (\ref{I(nu,x)_resultado}), (\ref%
{2F2_reduction_m}), (\ref{E1(z)=Ein(z) -ln(z)-gamma}),\ and (\ref%
{Sum_BesselK_n+1/2}) for $x>0$, we arrive at
\begin{eqnarray}
&&I_{3}^{\ast }\left( 0,n+\frac{1}{2};x\right) =\int_{0}^{\infty }e^{-xt}%
\left[ t\left( 1+t\right) \right] ^{n}\ln \left[ t\left( 1+t\right) \right]
dt  \label{I*3(0,n+1/2,x)} \\
&=&\sum_{k=0}^{n}\binom{n}{k}\int_{0}^{\infty }e^{-xt}t^{n+k}\ln \left[
t\left( 1+t\right) \right] dt=\sum_{k=0}^{n}\binom{n}{k}\mathcal{I}%
_{+}\left( n+k,x\right)  \notag \\
&=&\frac{n!e^{x/2}K_{n+1/2}\left( \frac{x}{2}\right) }{\sqrt{\pi }\,x^{n+1/2}%
}\left[ \mathrm{E}_{1}\left( x\right) -\gamma -\ln x\right] +\frac{n!}{%
x^{n+1}}\sum_{k=0}^{n}\frac{\left( n+k\right) !\,x^{-k}}{k!\left( n-k\right)
!}  \notag \\
&&\left\{ H_{n+k}+\left( -1\right) ^{n+k+1}\,\Gamma \left( -n-k,x\right)
\,\gamma \left( n+k+1,-x\right) -\sum_{\ell =1}^{k}\binom{k}{\ell }\left(
-x\right) ^{-\ell }\gamma \left( \ell ,x\right) \right\} .  \notag
\end{eqnarray}%
Take $\kappa =0$ and $\mu =n+\frac{1}{2}$ in (\ref{DmW_I*3}), and substitute
(\ref{I*3(0,n+1/2,x)})\ and (\ref{W_0,n+1/2}) in order\ to arrive at (\ref%
{DmW_0_n+1/2}), as we wanted to prove.
\end{proof}

Table \ref{Table_3B} shows $\mathrm{W}_{\kappa ,\mu }\left( x\right) $ with
respect to parameter $\mu $ for some particular values of $\kappa $ and $\mu
$, which has been calculated from (\ref{DmW_m_limt}), (\ref{DmW_n/2}), and (%
\ref{DmW_0_n+1/2}) with the aid of MATHEMATICA\ program.

\begin{center}
\begin{table}[tbp] \centering%
\caption{Derivative of $W_{\kappa,\mu}$ with respect
to $\mu$ by using (\ref{DmW_n/2}) and (\ref{DmW_0_n+1/2}).}%
\begin{tabular}{|c|c|c|}
\hline
$\kappa $ & $\mu $ & $\frac{\partial \mathrm{W}_{\kappa ,\mu }\left(
x\right) }{\partial \mu }$ \\ \hline\hline
$0$ & $\pm \frac{1}{2}$ & $\pm e^{x/2}\left[ \mathrm{Shi}\left( x\right) -%
\mathrm{Chi}\left( x\right) \right] $ \\ \hline
$0$ & $\pm \frac{3}{2}$ & $\pm x^{-1}e^{-x/2}\left\{ e^{x}\left( x-2\right) %
\left[ \mathrm{Shi}\left( x\right) -\mathrm{Chi}\left( x\right) \right]
+4\right\} $ \\ \hline
$0$ & $\pm \frac{5}{2}$ & $\pm x^{-2}e^{-x/2}\left\{ 4\left( x+8\right)
-e^{x}\left( x^{2}-6x+12\right) \left[ \mathrm{Shi}\left( x\right) -\mathrm{%
Chi}\left( x\right) \right] \right\} $ \\ \hline
$\frac{1}{2}$ & $\pm 1$ & $\pm x^{-1/2}e^{-x/2}\left\{ e^{x}\left[ \mathrm{%
Shi}\left( x\right) -\mathrm{Chi}\left( x\right) \right] +2\right\} $ \\
\hline
$\frac{1}{2}$ & $0$ & $0$ \\ \hline
$1$ & $\pm \frac{1}{2}$ & $\pm e^{-x/2}$ \\ \hline
$1$ & $\pm \frac{3}{2}$ & $\pm x^{-1}e^{-x/2}\left\{ 2\,e^{x}\left[ \mathrm{%
Shi}\left( x\right) -\mathrm{Chi}\left( x\right) \right] +3\left( x+2\right)
\right\} $ \\ \hline
$\frac{3}{2}$ & $\pm 1$ & $\pm x^{-1/2}e^{-x/2}\left( 2x+1\right) $ \\ \hline
$\frac{3}{2}$ & $\pm 2$ & $\pm x^{-3/2}e^{-x/2}\left\{ 2\left(
2x^{2}+7x+11\right) -6\,e^{x}\left[ \mathrm{Shi}\left( x\right) -\mathrm{Chi}%
\left( x\right) \right] \right\} $ \\ \hline
$2$ & $\pm \frac{3}{2}$ & $\pm e^{-x/2}\left( 3x+3+\frac{2}{x}\right) $ \\
\hline
$2$ & $\pm \frac{5}{2}$ & $\pm x^{-2}e^{-x/2}\left\{ 5\left(
x^{3}+5x^{2}+14x+20\right) -24\,e^{x}\left[ \mathrm{Shi}\left( x\right) -%
\mathrm{Chi}\left( x\right) \right] \right\} $ \\ \hline
\end{tabular}%
\label{Table_3B}%
\end{table}%
\end{center}

\section{Integral Whittaker functions $\mathrm{Wi}_{\protect\kappa ,\protect%
\mu }$ and $\mathrm{wi}_{\protect\kappa ,\protect\mu }$}

In \cite{apelblat2021integral}, we found some reduction formulas for the
integral Whittaker function $\mathrm{Wi}_{\kappa ,\mu }\left( x\right) $.
Next, we derive some new reduction formulas for $\mathrm{Wi}_{\kappa ,\mu
}\left( x\right) $ and $\mathrm{wi}_{\kappa ,\mu }\left( x\right) $ from
reduction formulas of the Whittaker function $\mathrm{W}_{\kappa ,\mu
}\left( x\right) $.

\begin{theorem}
The following reduction formula holds true for $n=0,1,2,\ldots $ and $\kappa
>0$:%
\begin{equation}
\mathrm{Wi}_{\kappa +n,\kappa -1/2}\left( x\right) =\left( -1\right)
^{n}\left( 2\kappa \right) _{n}\,2^{\kappa }\sum_{m=0}^{n}\binom{n}{m}\frac{%
\left( -2\right) ^{m}}{\left( 2\kappa \right) _{m}}\gamma \left( \kappa
+m,x/2\right) .  \label{Wi_k+n,k-1/2}
\end{equation}
\end{theorem}

\begin{proof}
According to \cite[Eqn. 13.18.17]{olver2010nist}
\begin{equation}
\mathrm{W}_{\kappa +n,\kappa -1/2}\left( x\right) =\left( -1\right)
^{n}n!\,e^{-x/2}x^{\kappa }L_{n}^{\left( 2\kappa -1\right) }\left( x\right) ,
\label{W_k+n,k-1/2}
\end{equation}%
where \cite[Eqn. 4.17.2]{lebedev1965special}%
\begin{equation}
L_{n}^{\left( \alpha \right) }\left( x\right) =\sum_{m=0}^{n}\frac{\Gamma
\left( n+\alpha +1\right) }{\Gamma \left( m+\alpha +1\right) }\frac{\left(
-x\right) ^{m}}{m!\left( n-m\right) !},  \label{Laguerre_def}
\end{equation}%
denotes the Laguerre polynomials. Insert (\ref{Laguerre_def})\ in (\ref%
{W_k+n,k-1/2})\ and integrate term by term according to the definition of
the integral Whittaker function (\ref{Wi_def}), to get%
\begin{eqnarray*}
&&\mathrm{Wi}_{\kappa +n,\kappa -1/2}\left( x\right) \\
&=&\left( -1\right) ^{n}\left( 2\kappa \right) _{n}\sum_{m=0}^{n}\binom{n}{m}%
\frac{\left( -1\right) ^{m}}{\left( 2\kappa \right) _{m}}%
\int_{0}^{x}e^{-t/2}t^{\kappa +m-1}dt.
\end{eqnarray*}%
Finally, take into account the defintion of the lower incomplete gamma
function \cite[Eqn. 8.2.1]{olver2010nist}:%
\begin{equation}
\gamma \left( \nu ,z\right) =\int_{0}^{z}t^{\nu -1}e^{-t}dt,\qquad \mathrm{Re%
}\,\nu >0,  \label{gamma_def}
\end{equation}%
and simplify the result to arrive at (\ref{Wi_k+n,k-1/2}), as we wanted to
prove.
\end{proof}

\begin{remark}
Taking $n=0$ in (\ref{Wi_k+n,k-1/2}), we recover the formula given in \cite%
{apelblat2021integral}.
\end{remark}

\begin{theorem}
The following reduction formula holds true for $x>0$, $n=0,1,2,\ldots $ and $%
\kappa \in
\mathbb{R}
$:%
\begin{equation}
\mathrm{wi}_{\kappa +n,\kappa -1/2}\left( x\right) =\left( -1\right)
^{n}\left( 2\kappa \right) _{n}\,2^{\kappa }\sum_{m=0}^{n}\binom{n}{m}\frac{%
\left( -2\right) ^{m}}{\left( 2\kappa \right) _{m}}\Gamma \left( \kappa
+m,x/2\right) ,  \label{wi_k+n,k-1/2}
\end{equation}%
where $\Gamma \left( \nu ,z\right) $ denotes the upper incomplete gamma
function (\ref{Gamma_def}).
\end{theorem}

\begin{proof}
Follow similar steps as in the previous theorem, but consider the definition
of the upper incomplete gamma function \cite[Eqn. 8.2.2]{olver2010nist}:%
\begin{equation}
\Gamma \left( \nu ,z\right) =\int_{z}^{\infty }t^{\nu -1}e^{-t}dt.
\label{Gamma_def}
\end{equation}
\end{proof}

\begin{theorem}
The following reduction formula holds true for $x>0$, and $n=0,1,2,\ldots $:%
\begin{equation}
\mathrm{wi}_{0,n+1/2}\left( x\right) =\sum_{m=0}^{n}\frac{\left( n+k\right)
!2^{-k}}{k!\left( n-k\right) !}\Gamma \left( -k,x/2\right) .
\label{wi_0,n+1/2}
\end{equation}
\end{theorem}

\begin{proof}
From (\ref{Sum_BesselK_n+1/2}) and (\ref{W_0,n+1/2}), we have
\begin{equation*}
\mathrm{W}_{0,n+1/2}\left( z\right) =e^{-z/2}\sum_{k=0}^{n}\frac{\left(
n+k\right) !\,z^{-k}}{k!\left( n-k\right) !},
\end{equation*}%
thus, integrating term by term, we obtain
\begin{equation*}
\mathrm{wi}_{0,n+1/2}\left( x\right) =\sum_{k=0}^{n}\frac{\left( n+k\right) !%
}{k!\left( n-k\right) !}\int_{x}^{\infty }e^{-t/2}\,t^{-k-1}dt.
\end{equation*}%
Finally, taking into account (\ref{Gamma_def}), we arrive at (\ref%
{wi_0,n+1/2}), as we wanted to prove.
\end{proof}

\begin{theorem}
For $x>0$ and $\mathrm{Re}\left( \frac{1}{2}+\mu -\kappa \right) >0$, the
following integral representation holds true:%
\begin{equation}
\mathrm{wi}_{\kappa ,\mu }\left( x\right) =\frac{1}{\Gamma \left( \frac{1}{2}%
+\mu -\kappa \right) }\int_{0}^{\infty }\frac{t^{\mu -\kappa -1/2}\left(
1+t\right) ^{\mu +\kappa -1/2}}{\left( \frac{1}{2}+t\right) ^{\mu +1/2}}%
\Gamma \left( \frac{1}{2}+\mu ,x\left( t+\frac{1}{2}\right) \right) dt.
\label{wi_int}
\end{equation}
\end{theorem}

\begin{proof}
According to (\ref{wi__def})\ and (\ref{W_k,mu_int_1}), we have
\begin{eqnarray*}
&&\mathrm{wi}_{\kappa ,\mu }\left( x\right) \\
&=&\frac{1}{\Gamma \left( \mu -\kappa +\frac{1}{2}\right) }\int_{x}^{\infty
}dt\,\,t^{\mu -1/2}e^{-t/2}\int_{0}^{\infty }e^{-x\,\xi }\xi ^{\mu -\kappa
-1/2}\left( 1+\xi \right) ^{\mu +\kappa -1/2}d\xi .
\end{eqnarray*}%
Exchange the integration order and calculate the inner integral using (\ref%
{Gamma_def}), to arrive at (\ref{wi_int}), as we wanted to prove.
\end{proof}

\begin{remark}
It is worth noting that we cannot follow the above steps to derive the
integral representation of $\mathrm{Wi}_{\kappa ,\mu }\left( x\right) $
because the corresponding integral does not converge, except for some
special cases such as the ones given in (\ref{Wi_k+n,k-1/2}).
\end{remark}

\begin{theorem}
For $x>0$ and $\mathrm{Re}\left( \frac{1}{2}+\mu -\kappa \right) >0$, the
following integral representation holds true:%
\begin{eqnarray}
&&\frac{\partial \mathrm{wi}_{\kappa ,\mu }\left( x\right) }{\partial \kappa
}=\frac{1}{\Gamma \left( \frac{1}{2}+\mu -\kappa \right) }  \label{Dkwi} \\
&&\int_{0}^{\infty }\left[ \psi \left( \frac{1}{2}+\mu -\kappa \right) +\ln
\left( \frac{1+t}{t}\right) \right] \frac{t^{\mu -\kappa -1/2}\left(
1+t\right) ^{\mu +\kappa -1/2}}{\left( \frac{1}{2}+t\right) ^{\mu +1/2}}%
\Gamma \left( \frac{1}{2}+\mu ,x\left( t+\frac{1}{2}\right) \right) dt.
\notag
\end{eqnarray}
\end{theorem}

\begin{proof}
Direct differentiation of (\ref{wi_int})\ with respect to $\kappa $ yields (%
\ref{Dkwi}), as we wanted to prove.
\end{proof}

\section{Conclusions}

The Whittaker function $\mathrm{W}_{\kappa ,\mu }\left( x\right) $ is
defined in terms of the Tricomi function, hence its derivative with respect
to the parameters $\kappa $ and $\mu $ can be expressed as infinite sums of
quotients of the digamma and gamma functions. Also, the parameter
differentiation of some integral representations of $\mathrm{W}_{\kappa ,\mu
}\left( x\right) $ leads to infinite integrals of elementary functions.
These sums and integrals has been calculated for some particular cases of
the parameters $\kappa $ and $\mu $ in closed-form. As an application of
these results, we have calculated an infinite integral containing the
Macdonald function. It is worth noting that all the results presented in
this paper has been both numerically and symbolically checked with
MATHEMATICA program.

In the first Appendix, we calculate a reduction formula for the first
derivative of the Kummer function, i.e. $G^{\left( 1\right) }\left(
a;a;z\right) $, which it is necessary for the derivation of Theorem \ref%
{Theorem_1}.

In the second Appendix, we calculate a reduction formula of the
hypergeometric function $_{2}F_{2}\left( 1,1;2,2+m;x\right) $ for
non-negative integer $m$, since it is not found in most common literature,
such as \cite{prudnikov1986integrals}. This reduction formula is used
throughout Section \ref{Section: Integral representations} in order to
simplify the results obtained.

Finally, we collect some reduction formulas for the Whittaker function $%
\mathrm{W}_{\kappa ,\mu }\left( x\right) $ in the last Appendix.

\appendix{}

\section{Calculation of $G^{\left( 1\right) }\left( a;a;z\right) $}

\begin{theorem}
The following reduction formula holds true:%
\begin{equation}
G^{\left( 1\right) }\left( \left.
\begin{array}{c}
a \\
a%
\end{array}%
\right\vert x\right) =\frac{x\,e^{x}}{a}\,_{2}F_{2}\left( \left.
\begin{array}{c}
1,1 \\
a+1,2%
\end{array}%
\right\vert -x\right) .  \label{G(1)_(a;a;x)}
\end{equation}
\end{theorem}

\begin{proof}
According to the definition of the Kummer function (\ref{1F1_Whittaker_def}%
), we have%
\begin{equation}
_{1}F_{1}\left( \left.
\begin{array}{c}
b \\
a%
\end{array}%
\right\vert x\right) =1+\sum_{n=0}^{\infty }\frac{\left( b\right) _{n+1}}{%
\left( a\right) _{n+1}}\frac{x^{n+1}}{\left( n+1\right) !}.  \label{1F1_(a)}
\end{equation}%
Taking into account \cite[Eqn. 18:5:7]{oldham2009atlas} \
\begin{equation*}
\left( \alpha \right) _{n+1}=\alpha \left( \alpha +1\right) _{n},
\end{equation*}%
and the definition of the generalized hypergeometric function (\ref{pFq_def}%
), we may recast (\ref{1F1_(a)})\ as
\begin{equation*}
_{1}F_{1}\left( \left.
\begin{array}{c}
b \\
a%
\end{array}%
\right\vert x\right) =1+\frac{b}{a}x\,_{2}F_{2}\left( \left.
\begin{array}{c}
1,b+1 \\
2,a+1%
\end{array}%
\right\vert x\right) ,
\end{equation*}%
thus, for $b\neq 0$, we obtain\footnote{%
It is worth noting that there is a typo in \cite[Eqn. 7.12.1(5)]%
{prudnikov1986integrals}.}%
\begin{equation}
_{2}F_{2}\left( \left.
\begin{array}{c}
1,b+1 \\
2,a+1%
\end{array}%
\right\vert x\right) =\frac{a}{b\,x}\left[ _{1}F_{1}\left( \left.
\begin{array}{c}
b \\
a%
\end{array}%
\right\vert x\right) -1\right] .  \label{2F2->1F1}
\end{equation}%
Applying L'H\^{o}pital's rule, calculate the limit $b\rightarrow 0$ in (\ref%
{2F2->1F1}), considering the notation given in (\ref{G(1)_def}),
\begin{equation}
_{2}F_{2}\left( \left.
\begin{array}{c}
1,1 \\
2,a+1%
\end{array}%
\right\vert x\right) =\frac{a}{x}G^{\left( 1\right) }\left( \left.
\begin{array}{c}
0 \\
a%
\end{array}%
\right\vert x\right) .  \label{2F2_G(1)}
\end{equation}%
Finally, differentiate Kummer's transformation formula (\ref%
{Kummer_transform}) with respect to the first parameter to obtain:%
\begin{equation}
G^{\left( 1\right) }\left( \left.
\begin{array}{c}
b \\
a%
\end{array}%
\right\vert x\right) =-e^{x}\,G^{\left( 1\right) }\left( \left.
\begin{array}{c}
b-a \\
b%
\end{array}%
\right\vert -x\right) .  \label{G(1)_transform}
\end{equation}%
Apply (\ref{G(1)_transform})\ in order to rewrite (\ref{2F2_G(1)})\ as (\ref%
{G(1)_(a;a;x)}), as we wanted to prove.
\end{proof}

\section{Calculation of $_{2}F_{2}\left( 1,1;2,2+m;x\right) $}

\begin{theorem}
For $m=0,1,2,\ldots $, the following reduction formula holds true:%
\begin{equation}
_{2}F_{2}\left( \left.
\begin{array}{c}
1,1 \\
2,2+m%
\end{array}%
\right\vert x\right) =\frac{m+1}{x}\left\{ H_{m}-\mathrm{Ein}\left(
-x\right) +\sum_{k=1}^{m}\binom{m}{k}x^{-k}\gamma \left( k,-x\right)
\right\} ,  \label{2F2_reduction}
\end{equation}%
where $\mathrm{Ein}\left( z\right) $ denotes the complementary exponential
integral.
\end{theorem}

\begin{proof}
Consider the function%
\begin{equation*}
R_{m}\left( x\right) =\frac{1}{m!}\,_{2}F_{2}\left( \left.
\begin{array}{c}
1,1 \\
2,1+m%
\end{array}%
\right\vert x\right) =\sum_{k=0}^{\infty }\frac{x^{k}}{\left( m+k\right)
!\left( k+1\right) },
\end{equation*}%
thus%
\begin{equation*}
\frac{d}{dx}\left[ x^{m}R_{m}\left( x\right) \right] =x^{m-1}R_{m-1}\left(
x\right) ,
\end{equation*}%
and by induction%
\begin{equation*}
\frac{d^{m}}{dx^{m}}\left[ x^{m}R_{m}\left( x\right) \right] =R_{0}\left(
x\right) =\frac{1}{x}\sum_{k=0}^{\infty }\frac{x^{k+1}}{\left( k+1\right) !}=%
\frac{e^{x}-1}{x}.
\end{equation*}%
Now, apply the repeated integral formula \cite[Eqn. 1.4.31]{olver2010nist}%
\begin{equation*}
f^{\left( -n\right) }\left( x\right) =\frac{1}{\left( n-1\right) !}%
\int_{0}^{x}\left( x-t\right) ^{n-1}f\left( t\right) dt,
\end{equation*}%
to obtain%
\begin{eqnarray}
R_{m+1}\left( x\right) &=&\frac{1}{\left( m+1\right) !}\,_{2}F_{2}\left(
\left.
\begin{array}{c}
1,1 \\
2,2+m%
\end{array}%
\right\vert x\right)  \notag \\
&=&\frac{x^{-m-1}}{m!}\int_{0}^{x}\left( x-t\right) ^{m}\left( \frac{e^{t}-1%
}{t}\right) dt.  \label{Repeated_integral_binomial}
\end{eqnarray}%
Use the binomial theorem to expand (\ref{Repeated_integral_binomial})\ as%
\begin{eqnarray}
&&_{2}F_{2}\left( \left.
\begin{array}{c}
1,1 \\
2,2+m%
\end{array}%
\right\vert x\right)  \label{2F2_a} \\
&=&\frac{m+1}{x}\left\{ \int_{0}^{x}\frac{e^{t}-1}{t}\,dt+\sum_{k=1}^{m}%
\binom{m}{k}x^{-k}\left( -1\right) ^{k}\int_{0}^{x}t^{k-1}\left(
e^{t}-1\right) dt\right\} .  \notag
\end{eqnarray}%
According to \cite[Eqn. 6.2.3]{olver2010nist}, we have%
\begin{equation}
\int_{0}^{x}\frac{e^{t}-1}{t}\,dt=-\mathrm{Ein}\left( -x\right) .
\label{Ein(x)}
\end{equation}%
Also, taking into account the definition of the lower incomplete gamma
function \cite[Eqn. 45:3:1]{oldham2009atlas}, we calculate for $k=1,2,\ldots
$%
\begin{equation}
\int_{0}^{x}t^{k-1}\left( e^{t}-1\right) dt=\left( -1\right) ^{k}\gamma
\left( k,-x\right) -\frac{x^{k}}{k}.  \label{Integral_gamma}
\end{equation}%
Therefore, substituting (\ref{Ein(x)}) and (\ref{Integral_gamma})\ in (\ref%
{2F2_a}), we have
\begin{equation*}
_{2}F_{2}\left( \left.
\begin{array}{c}
1,1 \\
2,2+m%
\end{array}%
\right\vert x\right) =\frac{m+1}{x}\left\{ -\mathrm{Ein}\left( -x\right)
+\sum_{k=1}^{m}\binom{m}{k}\left[ x^{-k}\gamma \left( k,-x\right) +\frac{%
\left( -1\right) ^{k+1}}{k}\right] \right\} .
\end{equation*}%
Finally, consider the formula \cite[Eqn. 0.155.4]{gradstein2015table}%
\begin{equation*}
\sum_{k=1}^{m}\binom{m}{k}\frac{\left( -1\right) ^{k+1}}{k}=H_{m},
\end{equation*}%
to arrive at (\ref{2F2_reduction}), as we wanted to prove
\end{proof}

\section{Reduction formulas for the Whittaker function $\mathrm{W}_{\protect%
\kappa ,\protect\mu }\left( x\right) $\label{Appendix_Reduction_Whittaker}}

For convenience of the readers, reduction formulas for the Whittaker
function $\mathrm{W}_{\kappa ,\mu }\left( x\right) $ are presented in their
explicit form in Table \ref{Table_5}.

\begin{center}
\begin{table}[htbp] \centering%
\caption{Whittaker function $W_{\kappa,\mu}(x)$  for particular values of $\kappa$ and
$\mu$.}%
\begin{tabular}{|c|c|c|}
\hline
$\kappa $ & $\mu $ & $\mathrm{W}_{\kappa ,\mu }\left( x\right) $ \\
\hline\hline
$-\frac{1}{4}$ & $\pm \frac{1}{4}$ & $\sqrt{\pi }e^{x/2}x^{1/4}\mathrm{erfc}%
\left( \sqrt{x}\right) $ \\ \hline
$-\frac{1}{2}$ & $\pm \frac{1}{2}$ & $\frac{x}{\sqrt{\pi }}\left[
K_{1}\left( \frac{x}{2}\right) -K_{0}\left( \frac{x}{2}\right) \right] $ \\
\hline
$-\frac{1}{2}$ & $\pm \frac{1}{6}$ & $3\frac{x}{\sqrt{\pi }}\left[
K_{2/3}\left( \frac{x}{2}\right) -K_{1/3}\left( \frac{x}{2}\right) \right] $
\\ \hline
$-\frac{1}{2}$ & $\pm 1$ & $x^{-1/2}e^{-x/2}$ \\ \hline
$0$ & $0$ & $\sqrt{\frac{x}{\pi }}\,K_{0}\left( \frac{x}{2}\right) $ \\
\hline
$0$ & $\pm \frac{1}{2}$ & $e^{-x/2}$ \\ \hline
$0$ & $\pm 1$ & $\sqrt{\frac{x}{\pi }}\,K_{1}\left( \frac{x}{2}\right) $ \\
\hline
$0$ & $\pm \frac{3}{2}$ & $x^{-1}e^{-x/2}\left( x+2\right) $ \\ \hline
$0$ & $\pm \frac{5}{2}$ & $x^{-2}e^{-x/2}\left( x^{2}+6\,x+12\right) $ \\
\hline
$\frac{1}{4}$ & $\pm \frac{1}{4}$ & $x^{1/4}e^{-x/2}$ \\ \hline
$\frac{1}{2}$ & $\pm \frac{1}{6}$ & $\frac{x}{2\sqrt{\pi }}\left[
K_{1/3}\left( \frac{x}{2}\right) +K_{2/3}\left( \frac{x}{2}\right) \right] $
\\ \hline
$\frac{1}{2}$ & $\pm \frac{1}{4}$ & $\frac{x}{2\sqrt{\pi }}\left[
K_{1/4}\left( \frac{x}{2}\right) +K_{3/4}\left( \frac{x}{2}\right) \right] $
\\ \hline
$\frac{1}{2}$ & $\pm \frac{1}{2}$ & $\frac{x}{2\sqrt{\pi }}\left[
K_{0}\left( \frac{x}{2}\right) +K_{1}\left( \frac{x}{2}\right) \right] $ \\
\hline
$\frac{1}{2}$ & $\pm 1$ & $\,x^{-1/2}e^{-x/2}\,\left( x+1\right) $ \\ \hline
$\frac{1}{2}$ & $\pm 2$ & $x^{-3/2}e^{-x/2}\,\left( x^{2}+4\,x+6\right) $ \\
\hline
$1$ & $\pm \frac{3}{2}$ & $x^{-1}e^{-x/2}\left( x^{2}+2\,x+2\right) $ \\
\hline
$1$ & $\pm 1$ & $\frac{1}{2}\sqrt{\frac{x}{\pi }}\left[ x\,K_{0}\left( \frac{%
x}{2}\right) +\left( x+1\right) K_{1}\left( \frac{x}{2}\right) \right] $ \\
\hline
$1$ & $\pm 2$ & $\frac{1}{2\sqrt{\pi \,x}}\left[ x\left( x+3\right)
K_{0}\left( \frac{x}{2}\right) +\left( x^{2}+4\,x+12\right) K_{1}\left(
\frac{x}{2}\right) \right] $ \\ \hline
$2$ & $\pm 2$ & $\frac{1}{4\sqrt{\pi \,x}}\left[ x\left(
2\,x^{2}+2x+3\right) K_{0}\left( \frac{x}{2}\right) +2\left(
x^{3}+2\,x^{2}+4\,x+6\right) K_{1}\left( \frac{x}{2}\right) \right] $ \\
\hline
\end{tabular}%
\label{Table_5}%
\end{table}%
\end{center}

\pagebreak

\bibliographystyle{plain}

\end{document}